\documentclass{article}
\usepackage[toc,page]{appendix}
\usepackage[utf8]{inputenc}
\usepackage[english]{babel}
\DeclareMathAlphabet{\mathscr}{OT1}{pzc}{m}{it} 
\usepackage{amsmath}
\usepackage{bbold}
\usepackage{dsfont} 
\usepackage{amssymb}
\usepackage[T1]{fontenc}
\usepackage{mathtools}   
\usepackage{amsfonts}
\usepackage{bbm}
\usepackage{mathrsfs}  
\usepackage{graphicx}
\usepackage{relsize}
\usepackage{enumitem}
\usepackage{amssymb,amsmath,amsthm}
\numberwithin{equation}{section}
\usepackage{authblk}

\newtheorem{theorem}{Theorem}[section]
\newtheorem{notation}[theorem]{Notation}
\newtheorem{lemma}[theorem]{Lemma}

\newtheorem{proposition}[theorem]{Proposition}
\newtheorem{corollary}[theorem]{Corollary}
\newtheorem{definition}[theorem]{Definition}
\newtheorem{hypothesis}[theorem]{Hypothesis}
\newtheorem{remark}[theorem]{Remark}
\newenvironment{prooff}[1]{\begin{trivlist}
\item {\it \bf Proof}\quad} {\qed\end{trivlist}}
\usepackage{amsmath,amssymb}
\makeatletter
\newsavebox\myboxA
\newsavebox\myboxB
\newlength\mylenA
\newcommand*\xoverline[2][0.75]{%
    \sbox{\myboxA}{$\m@th#2$}%
    \setbox\myboxB\null
    \ht\myboxB=\ht\myboxA%
    \dp\myboxB=\dp\myboxA%
    \wd\myboxB=#1\wd\myboxA
    \sbox\myboxB{$\m@th\overline{\copy\myboxB}$}
    \setlength\mylenA{\the\wd\myboxA}
    \addtolength\mylenA{-\the\wd\myboxB}%
    \ifdim\wd\myboxB<\wd\myboxA%
       \rlap{\hskip 0.5\mylenA\usebox\myboxB}{\usebox\myboxA}%
    \else
        \hskip -0.5\mylenA\rlap{\usebox\myboxA}{\hskip 0.5\mylenA\usebox\myboxB}%
    \fi}
\makeatother

\title{G\^ateaux type path-dependent PDEs and
 BSDEs with Gaussian forward processes}

\author{
Adrien BARRASSO \thanks{ENSTA Paris and Ecole Polytechnique,
 Institut Polytechnique de Paris 828.
E-mail: \sf adrien.barrasso@polytechnique.edu \\ }
\qquad\quad
Francesco RUSSO\thanks{ENSTA Paris, Institut Polytechnique
de Paris, Unit\'e de Math\'ematiques appliqu\'ees, 828, boulevard des Mar\'echaux, F-91120 Palaiseau, France. E-mail:
 \sf francesco.russo@ensta-paris.fr}}

\date{June 26th 2019}

\begin{document}
\maketitle
{\bf Abstract.}
We are interested in path-dependent semilinear PDEs, where the derivatives are of G\^ateaux type in specific directions $k$ and $b$, being the kernel
 functions of
 a Volterra Gaussian process $X$. Under some conditions on $k, b$ and the coefficients of 
the PDE, we prove existence and uniqueness of a decoupled mild solution, 
a notion  introduced in a previous paper by the authors.
We also show that the solution of the PDE can
 be represented through BSDEs where the forward (underlying) process is $X$.

\bigskip
{\bf MSC 2010} Classification. 
60G15; 60H30; 35S05; 60J35; 45D05 


\section{Introduction}

Backward SDEs (in short BSDEs) are naturally linked to non-linear deterministic evolution equations. In one of their  pioneering work \cite{pardoux1992backward},  Pardoux and Peng showed that Markovian BSDEs  for which the randomness comes from an underlying which is the solution of a classical SDE, are linked to classical semilinear PDEs.
In this framework an impressive amount of papers has been produced.

In the recent times, particular attention was devoted to the
case when the driver and terminal condition of the BSDE depend on the whole path of the forward underlying
 process
 which can be a Brownian motion.
 Those are of type
\begin{equation} \label{C7EBSDE}
Y^{s,\eta}=\xi\left((B^{s,\eta}_t)_{t\in[0,T]}\right)+\int_{\cdot}^Tf\left(r,(B^{s,\eta}_t)_{t\in[0,r]},Y^{s,\eta}_r,Z^{s,\eta}_r\right)dr-\int_{\cdot}^TZ^{s,\eta}_rdB^{s,\eta}_r,
\end{equation}
 where $B$ is a Brownian motion and for any $s\in[0,T]$, $\eta\in\mathbbm{D}([0,T],\mathbbm{R}^d)$, $B^{s,\eta}=\eta(\cdot\wedge s)+(B_{\cdot\vee s}-B_s)$.
If in the Markovian case those were related to usual 
PDEs, in the present path-dependent framework,
those were linked to the so called path-dependent PDEs (see for instance \cite{Peng2016,etzI}) of the form
\begin{equation}\label{C7PDEparabolique}
\left\{
\begin{array}{l}
D \Phi + \frac{1}{2}Tr(\nabla^2 \Phi) + f(\cdot,\cdot,\Phi,\nabla \Phi)=0\quad \text{ on } [0,T]\times\mathbbm{D}([0,T],\mathbbm{R}^d) \\
\Phi_T = \xi.
\end{array}\right.
\end{equation}
There, $D$ (resp. $\nabla$) is the horizontal (resp. vertical) derivative introduced in \cite{dupire}.
For \eqref{C7PDEparabolique} the authors discussed classical or viscosity (probabilistic) type solution.
Variants of it, replacing the Brownian motion with the solution of path-dependent SDEs were
considered for instance by \cite{cosso_russo15a, cosso_russo15b} with a different formalism. 
\cite{cosso_russo15b} for instance introduced the notion of
stong-viscosity solution
(based on approximation techniques), which 
constitutes a purely analytic concept. 

Indeed such path-dependent PDEs have been investigated 
by several methods. For instance strict (classical, regular) solutions 
have been 
studied in \cite{DGR, flandoli_zanco13, cosso_russo15a}
under the point of view of Banach space valued stochastic processes.
Another interesting approach (probabilistic) but still based 
on approximation (discretizations) was given by
\cite{leao_ohashi_simas14}.
More recently, \cite{BionNadal} produced 
a viscosity solution to a more general path-dependent (possibly integro)-PDE
 through  dynamic risk measures. 
In all those cases
the solution $\Phi$ of \eqref{C7PDEparabolique} was associated 
to the component $Y^{s,\eta}$ of the solution couple
$(Y^{s,\eta}, Z^{s,\eta})$ of \eqref{C7EBSDE} with initial time $s$ and
initial condition $\eta$.
A challenging link to be explored was 
the link between $Z^{s,\eta}$ and the solution of the path-dependent PDE $\Phi$.
 For instance in the case of Fr\'echet $C^{0,1}$  solutions $\Phi$ 
defined on $C([0,T])$,
 then
 $Z^{s,\eta}$ is equal to the {\it ''vertical''} derivative
$\nabla \Phi$, see for instance \cite{masiero}. 


An important step forward concerning path-dependent PDEs associated with BSDEs involving a solution of a path-dependent SDEs including the possibility of jumps and coefficients which were not necessarily continuous
was done in \cite{paperPathDep}. 
The concept of solution was there the {\it decoupled mild solution} which is based
on semigroup type techniques. 
That notion, is competitive with the notion of viscosity solution, especially when such viscosity solutions do not necessarily exist. 
Moreover, that notion of solution
also provides a solution to the so called
{\it identification problem}, meaning that it links the second component 
$Z$ of the BSDE, to the PDE.

The natural question raised by this paper is the following.
What about the case when the Brownian motion $B$ is replaced with
 a  (non-Markovian, non-semimartingale) process  such as
fractional Brownian motion? The idea is to extend the consideration 
of  \cite{paperPathDep} to this framework.
The basic reference paper for this work is
\cite{ViensZh}, that considered for the first time a BSDE which forward process was the solution of a Volterra SDE. This includes the kind of Gaussian processes which we consider. They related this BSDE to a G\^ateaux type PDE close to \eqref{C7IntroPDE} by showing that if the PDE admits a classical solution, that solution provides a solution of the BSDE.
 Our work provides the converse implication. We start from 
the well-posedness of a class of BSDEs, and show that they produce, under very mild regularity assumptions on the coefficients, 
a {\it decoupled mild} solution to the path-dependent PDE.

Let $(\Omega,\mathcal{F})$ be the canonical space  where $\Omega$ is
the set $\mathcal{C}_0([0,T],\mathbbm{R}^d), d \ge 1,$
 of $\mathbbm{R}^d$-valued continuous functions on $[0,T]$ vanishing at $0$, equipped with its uniform norm and $\mathcal{F}$ its Borel $\sigma$-field. 
 

We fix $b:[0,T]\times[0,T]\rightarrow \mathbbm{R}^d$ and $k:[0,T]\times[0,T]\rightarrow \mathcal{M}_d(\mathbbm{R})$ some two parameters  functions such that for all $t\in[0,T]$, $b(\cdot,t),k(\cdot,t)$ vanish on $[0,t[$ and are continuous, admitting a right-derivative on $[t,T]$.

On the canonical space, we consider a Gaussian measure $\mathbbm{Q}$ under which there exists a $d$-dimensional Brownian motion $B$ such that the canonical process $X$ admits the representation
\begin{equation}
X_t=\int_0^tb(t,r)dr+\int_0^tk(t,r)dB_r,\quad t\in[0,T].
\end{equation} 
For every ''initial time and path'' $(s,\eta)$ we introduce the law 
$\mathbbm{Q}^{s,\eta}$ of $X$ conditioned by the fact that,
 on $[0,s]$,  $X$ coincides  with the path $\eta$.
 For every $(s,\eta)$, $\mathbbm{Q}^{s,\eta}$ is a Gaussian measure 
 of mean function $m_s[\eta]$, where $m_s$ is a continuous linear operator on $\Omega$. The reader can refer to 
\cite{bogachevGauss} concerning Gaussian measures and related notions,
see also Definition \ref{DefBasic}.

 We will  show that $(\mathbbm{Q}^{s,\eta})_{(s,\eta)\in[0,T]\times \Omega}$ defines what we call a \textbf{path-dependent canonical class},
see Definition \ref{DefCondSyst}, notion which was introduced by the authors in \cite{paperMPv2}.
This concept extends the well-known historical notion of  Markov canonical class to the path-dependent (therefore non-Markovian) setting.

Given this set of probability measures,
under every $\mathbbm{Q}^{s,\eta}$, we consider the BSDE (indexed by $(s,\eta)$)
\begin{equation} \label{C7IntroBSDE}
Y^{s,\eta}_{\cdot}=\xi(X)+\int_{\cdot}^{T}f\left(r,X,Y^{s,\eta}_r,\frac{d\langle M^{s,\eta},m^{T,s,\eta}\rangle_r}{dr}\right)dr-(M^{s,\eta}_T-M^{s,\eta}_{\cdot}),
\end{equation} 
where $m^{T,s,\eta}: t\longmapsto \mathbbm{E}^{s,\eta}[X_T|\mathcal{F}_t]$ is the  \textit{driving martingale} of the BSDE.
In the case when $k(t,\cdot)\equiv\mathds{1}_{[0,t]}$  and $b\equiv 0$ then this driving martingale $m^{T,s,\eta}$ is  $\mathbbm{P}^{s,\eta}$-a.s. equal to $X$  and is the conditioned Brownian motion $B^{s,\eta}$ appearing in \eqref{C7EBSDE}. This case was already considered in a more general framework, in  \cite{paperPathDep}.
The main aim of this paper is to study the path-dependent PDE which replaces  \eqref{C7PDEparabolique} when one 
considers the previous BSDE \eqref{C7IntroBSDE} instead of \eqref{C7EBSDE}.

Thanks to the theory which we have developed in \cite{paperMPv2}, we can associate to the family of probability measures $(\mathbbm{Q}^{s,\eta})_{(s,\eta)\in[0,T]\times \Omega}$
 what we call a \textbf{path-dependent system of projectors} $(P_s)_{s\in[0,T]}$, a notion which  replaces the one of Markovian semigroup.
We define the  linear operator $
\tilde A$, acting on a domain $\mathcal{D}(\tilde A)$ of functions $\tilde\Phi$ defined on $\mathbbm{D}([0,T],\mathbbm{R}^d))$  by
 \begin{equation}\label{IntroDefA}
   \tilde A(\tilde\Phi)_t:= D\tilde\Phi_{t}+\nabla_{b(\cdot,t)}\tilde\Phi_{t} +\frac{1}{2}\underset{i\leq d}{\sum}\nabla_{k_i(\cdot,t)}^2\tilde\Phi_{t},
\quad t\in[0,T].
 \end{equation} 

In \eqref{IntroDefA}, $\nabla_h$ (resp. $\nabla^2_{h,\ell}$) is the
 first (resp. second) order G\^ateaux type derivatives
in the direction $h$ (resp. $h, \ell$)
 and $D$ is a time derivative. Those operators act on  functionals defined on a set of cadlag functions. 
 Again when $k(t,\cdot)\equiv\mathds{1}_{[0,t]}$  and $b\equiv 0$ and if $\Phi_t(\omega) = \tilde{\Phi}_t(\omega^t)$, then $\nabla_{k(\cdot,t)}\tilde{\Phi}_t(\omega^t)=\nabla\Phi_t(\omega)$, where $\nabla$ is now the vertical derivative introduced in \cite{dupire}.

We introduce the {\it mean random field} $m = (m_s)_{s\in [0,T]}$,
where $m_s[\eta](t)$ is the conditional expectation of $X_t$
knowing that $X$ coincides with  $\eta$ on $[0,s]$, see Proposition 
\ref{CondExp}. 
In particular $m^T:(s,\eta)\longmapsto m_s[\eta](T)=\mathbbm{E}^{s,\eta}[X_T]$
 is called  \textbf{prediction martingale} in the literature, see for instance Remark 3.2 in \cite{Sottinen}.

 Then we introduce the operator $A$ on a certain domain $\mathcal{D}(A)$ which to each $\Phi = \tilde \Phi \circ m$
associates $(\tilde A (\tilde \Phi))\, \circ\, m$.
  We also introduce in Definition \ref{C7NotGamma}, the bilinear operator $\Gamma$ which to any $\Phi,\Psi\in\mathcal{D}(A)$ maps $A(\Phi\Psi)-\Phi A(\Psi)-\Psi A(\Phi)$. This operator was already introduced in another context in \cite{paperPathDep} and extends the \textit{carré du champ} operator appearing in the Markov processes literature, see \cite{dellmeyerD} for instance.

We show in Proposition \ref{PropWeakGen} that $A$
is a weak generator of $(P_s)_{s\in[0,T]}$, see Definition \ref{WeakGen}.
That operator $A$ is therefore linked to the probability measures $(\mathbbm{Q}^{s,\eta})_{(s,\eta)\in[0,T]\times \Omega}$ mentioned above, and this will lead us to show that the BSDE \eqref{C7IntroBSDE} permits to solve the following semilinear path-dependent PDE which we denote $PDPDE(f,\xi)$:

 \begin{equation}\label{C7IntroPDE}
\left\{
 \begin{array}{l}
 A  (\Phi) + f(\cdot,\cdot,\Phi,\Gamma(m^T,\Phi))= 0\text{ on }[0,T]\times \Omega \\
\Phi_{T} = \xi, \text{ on }[0,T]\times \Omega .
 \end{array}\right.
 \end{equation}

A process $Y$ will be called a \textbf{decoupled mild solution of} $PDPDE(f,\xi)$ if there exists an $\mathbbm{R}^d$-valued auxiliary  process $Z$ such that for all $(s,\eta)\in[0,T]\times\Omega$
  \begin{equation}\label{C7IntroAbMildEq}
		 \left\{
		 \begin{array}{rl}
		 Y_s(\eta)&=P_s[\xi](\eta)+\int_s^TP_s\left[f\left(r,\cdot,Y_r,Z_r\right)\right](\eta)dr\\
		 (Ym^T)_s(\eta) &=P_s[\xi X_T](\eta) -\int_s^TP_s\left[\left(Z_r-m^T_rf\left(r,\cdot,Y_r,Z_r\right)\right)\right](\eta)dr.
		 \end{array}\right.
		 \end{equation}
We emphasize that decoupled mild solutions were introduced  in the  framework of classical parabolic PDEs in \cite{paper3}, and in the path-dependent framework in \cite{paperPathDep}. Those extend the notion of {\it classical solution} i.e. a functional $\Phi$ in the domain $\mathcal {D}(A)$ fulfilling \eqref{C7IntroPDE}.

The main result of this paper is Theorem \ref{MainTheorem} which shows that when $\xi$ is measurable with polynomial growth and $f$ is measurable with polynomial growth in $\omega$ and uniformly Lipschitz in the last two variables,  then $PDPDE(f,\xi)$ admits a unique decoupled mild solution $Y$.

As anticipated, another feature of the paper is that the solution admits a probabilistic representation. Indeed, the unique decoupled mild solution of $PDPDE(f,\xi)$ is given by 
\begin{equation}
Y:(s,\eta)\longmapsto Y^{s,\eta}_s,
\end{equation}
where $Y^{s,\eta}$ is the solution of BSDE \eqref{C7IntroBSDE}.

When $b\equiv 0$ and $k(t,\cdot)\equiv \mathds{1}_{[0,t]}$ for all $t$, then for every $(s,\eta)$, $\mathbbm{Q}^{s,\eta}$ is the law of the "conditioned" Brownian motion $B^{s,\eta}$ introduced after \eqref{C7EBSDE}. In this case, our BSDE  \eqref{C7IntroBSDE} is simply \eqref{C7EBSDE} and \eqref{C7IntroPDE} becomes \eqref{C7PDEparabolique}. Existence and uniqueness of a decoupled mild solution in this case 
was already shown in our previous paper \cite{paperPathDep}.
That paper includes the case of (semimartingale-)solutions to path-dependent SDEs with jumps;
in that case the driving martingale of the BSDE is the martingale component of
the semimartingale $X$.


The paper is organized as follows. Section \ref{S1} recalls some notions and results related to path-dependent canonical classes and systems of projectors, which were introduced by the authors in \cite{paperMPv2}. Section \ref{SGauss} is mainly  devoted to showing that, under some conditions, a Gaussian measure 
 induces a path-dependent canonical class when considering its regular conditional probability distributions, see Proposition \ref{ClassGauss} for the centered case and Proposition \ref{ProbaDrift} for the case with a drift.
Section \ref{S4} is the main section of the paper. It introduces the path-dependent PDE for which we will show well-posedness, and the associated BSDE.
First, Section \ref{S41} introduces the assumptions on $k,b$, 
see Hypothesis \ref{Hypbk}. Then in Section \ref{S42} we define the linear operators $\tilde{A}$ and $A$ appearing in \eqref{IntroDefA} and \eqref{C7IntroPDE} with the corresponding domains, see Definition \ref{DefDA}. 
 Theorem \ref{ThViens} provides an  
  It\^o formula for elements of $\mathcal{D}(A)$. In Section \ref{S43} we introduce the driving martingale of the BSDE (see Notation \ref{mT}) and study its properties, see Proposition \ref{mTprop}.
Finally, in Section \ref{S44}, we consider the path-dependent PDE \eqref{C7AbstractEq} and show in Theorem \ref{MainTheorem} that, under Hypothesis \ref{C7HypBSDE}, it admits a unique decoupled mild solution
 and a probabilistic representation through the BSDE \eqref{C7BSDE}. Proposition \ref{C7classical} shows that any classical solution of \eqref{C7AbstractEq} is also a decoupled mild solution, and conversely that if the unique decoupled mild solution belongs to $\mathcal{D}(A)$ then it is quasi surely (see Definition \ref{C7zeropotential}) a classical solution.

\section{Preliminaries, path-dependent canonical classes and systems of projectors}\label{S1}

In this paper, we will make use of 
notions and  results concerning path-dependent canonical classes, which were introduced in Section 3 of\cite{paperMPv2}. We give here the main definitions and results related to that concept.

We start by fixing some basic vocabulary and notations.
\begin{notation}

	A topological space $E$ will always be considered as a measurable space 
	equipped with its Borel $\sigma$-field which shall be denoted $\mathcal{B}(E)$.
	
	Let $(\Omega,\mathcal{F})$, $(E,\mathcal{E})$ be two measurable spaces. A measurable mapping from $(\Omega,\mathcal{F})$ to $(E,\mathcal{E})$ shall often be called a \textbf{random variable} (with values in $E$), or in short r.v.

	Given a measurable space $\left(\Omega,\mathcal{F}\right)$, for any $p \ge 1$, the set of real valued random variables with finite $p$-th moment under probability $\mathbbm{P}$ will be denoted $\mathcal{L}^p(\mathbbm{P})$ or $\mathcal{L}^p$ if there can be no ambiguity concerning the underlying probability.
	
	Given a stochastic basis, for any cadlag locally square integrable martingales $M,N$, we denote  $\langle M,N\rangle$ (or simply $\langle M\rangle$ if $M=N$)  their (predictable) \textbf{angular bracket}.
\end{notation}

\begin{notation}\label{canonicalspace}
	We fix $T\in\mathbbm{R}_+^*$ and $d\in\mathbbm{N}^*$.
	$\Omega:=\mathcal{C}_0([0,T],\mathbbm{R}^d)$  will denote the   space of continuous functions from $[0,T]$ to $\mathbbm{R}^d$ vanishing at $0$.
	
	For every $t\in[0,T]$ we denote the coordinate mapping $X_t:\omega\mapsto\omega(t)$ and we define on $\Omega$ the $\sigma$-field  $\mathcal{F}:=\sigma(X_r|r\in[0,T])$. The coordinates of $X$ are denoted $X^1,\cdots,X^d$.
	On the measurable space $(\Omega,\mathcal{F})$, we introduce the \textbf{initial filtration}  $\mathbbm{F}^o:=(\mathcal{F}^o_t)_{t\in[0,T]}$, where $\mathcal{F}^o_t:=\sigma(X_r|r\in[0,t])$, and the (right-continuous) \textbf{canonical filtration} $\mathbbm{F}:=(\mathcal{F}_t)_{t\in[0,T]}$,
	where $\mathcal{F}_t:=\underset{s\in]t,T]}{\bigcap}\mathcal{F}^o_s$ if $t<T$ and $\mathcal{F}_T:=\mathcal{F}^o_T=\mathcal{F}$. 
	$\left(\Omega,\mathcal{F},\mathbbm{F}\right)$ will be called the \textbf{canonical space}, and $X$ the \textbf{canonical process}.
	On $[0,T]\times\Omega$, we will denote by $\mathcal{P}ro^o$ (resp. $\mathcal{P}re^o$) the  $\mathbbm{F}^o$-progressive (resp. $\mathbbm{F}^o$-predictable)  $\sigma$-field.
	$\Omega$ will be equipped with the sup norm $\|\cdot\|_{\infty}$ which makes it a Banach space, and for which the Borel $\sigma$-field is $\mathcal{F}$.
	
	$\mathcal{P}(\Omega)$ will denote the set of probability measures on $\Omega$ and will be equipped with the topology of weak convergence of measures which also makes it a Polish space being $\Omega$ itself Polish, see Theorems 1.7 and 3.1 in Chapter 3 of \cite{EthierKurz}. It will also be equipped with the associated Borel $\sigma$-field.
\end{notation}

\begin{notation}\label{Stopped}
	For any $\omega\in\Omega$ and $t\in[0,T]$, the path $\omega$ stopped at time $t$ $r\mapsto \omega(r\wedge t)$ will be denoted $\omega^t$.
\end{notation}

\begin{definition}\label{DefCondSyst}
	A \textbf{path-dependent canonical class} will be a set of probability measures $(\mathbbm{P}^{s,\eta})_{(s,\eta)\in[0,T]\times \Omega}$ defined on the canonical space $(\Omega,\mathcal{F})$. It will verify the three following items.
	\begin{enumerate}
		\item For every  $(s,\eta)\in[0,T]\times \Omega$, $\mathbbm{P}^{s,\eta}( \omega^s=\eta^s)=1$;
		\item for every $s\in[0,T]$ and $F\in\mathcal{F}$, the mapping
		\\
		$\begin{array}{ccl}
		\eta&\longmapsto& \mathbbm{P}^{s,\eta}(F)\\
		\Omega&\longrightarrow&[0,1]
		\end{array}$ is $\mathcal{F}^o_s$-measurable;
		\item for every  $(s,\eta)\in[0,T]\times \Omega$, $t\geq s$  and $F\in\mathcal{F}$,
		\begin{equation} \label{DE13}
		\mathbbm{P}^{s,\eta}(F|\mathcal{F}^o_t)(\omega)=\mathbbm{P}^{t,\omega}(F)\text{ for }\mathbbm{P}^{s,\eta}\text{ almost all }\omega.
		\end{equation}
	\end{enumerate} 
	This implies in particular that for every  $(s,\eta)\in[0,T]\times \Omega$ and $t\geq s$, then $(\mathbbm{P}^{t,\omega})_{\omega\in\Omega}$ is a regular 
	conditional expectation of $\mathbbm{P}^{s,\eta}$ by $\mathcal{F}^o_t$, see the definition above Theorem 1.1.6 in \cite{stroock} for instance.
	\\
	\\
	A path-dependent canonical class $(\mathbbm{P}^{s,\eta})_{(s,\eta)\in[0,T]\times \Omega}$ will be said to be \textbf{progressive} if for every $F\in\mathcal{F}$, the mapping 
	$(t,\omega)\longmapsto \mathbbm{P}^{t,\omega}(F)$ is $\mathbbm{F}^o$-progressively measurable.
\end{definition}
\begin{remark}\label{C7Borel}
	Given a path-dependent canonical class,
	one can easily show  by approximation through simple functions the
	following. Let  $Z$ be any random variable.
	\begin{itemize}	
		\item For any $s\in[0,T]$ 
then $\eta \longmapsto \mathbbm{E}^{s,\eta}[Z]$ is $\mathcal{F}^o_s$-measurable and for every  $(s,\eta)\in[0,T]\times \Omega$, $t\geq s$,
		$\mathbbm{E}^{s,\eta}(Z|\mathcal{F}^o_t)(\omega)=\mathbbm{E}^{t,\omega}(Z)\text{ for }\mathbbm{P}^{s,\eta}\text{ almost all }\omega$, provided
		previous expectations  are finite;
		\item 	if the  path-dependent canonical class is progressive,
		$(t,\omega)\longmapsto
		\mathbbm{E}^{t,\omega}[Z]$ is $\mathbbm{F}^o$-progressively measurable, provided
		previous expectations are finite.
	\end{itemize}
\end{remark}

Very often path-dependent canonical classes will verify the following important hypothesis which is a reinforcement of \eqref{DE13}.
\begin{hypothesis}\label{HypClass}
	For every  $(s,\eta)\in[0,T]\times \Omega$, $t\geq s$  and $F\in\mathcal{F}$,
	\begin{equation} \label{DE14}
	\mathbbm{P}^{s,\eta}(F|\mathcal{F}_t)(\omega)=\mathbbm{P}^{t,\omega}(F)\text{ for }\mathbbm{P}^{s,\eta}\text{ almost all }\omega.
	\end{equation}
\end{hypothesis}

\begin{notation} \label{N27}
	$\mathcal{B}_b(\Omega)$ stands for 
	the set of real bounded measurable functions on $\Omega$.
	Let $s\in[0,T]$, $\mathcal{B}^s_b(\Omega)$  will denote the set of real bounded $\mathcal{F}^o_s$-measurable functions on $\Omega$.
	We also denote by $\mathcal{B}^+_b(\Omega)$ the subset of r.v. $\phi\in\mathcal{B}_b(\Omega)$ such that $\phi(\omega)\geq 0$ for all $\omega\in\Omega$.
\end{notation}
\begin{definition}\label{DefCondOp}
	\begin{enumerate}\
		\item A linear map $Q: \mathcal{B}_b(\Omega) \rightarrow \mathcal{B}_b(\Omega)$
		is said {\bf positivity preserving monotone}  
		if for every $\phi\in\mathcal{B}^+_b(\Omega)$ then 
		$Q[\phi]\in\mathcal{B}^+_b(\Omega)$ and 
		for every increasing converging (in the pointwise sense) sequence 
		$f_n\underset{n}{\longrightarrow}f$ then   $Q[f_n]\underset{n}{\longrightarrow} Q[f]$  pointwise.
		\item 	A family $(P_s)_{s\in[0,T]}$ of positivity preserving monotone linear operators on $\mathcal{B}_b(\Omega)$ 
		will be called a \textbf{path-dependent system of projectors}  if it verifies the three following items.
		
		\begin{itemize}
			\item For all $s\in[0,T]$, the restriction of $P_s$ on $\mathcal{B}^s_b(\Omega)$
			coincides with the identity;
			\item for all $s\in[0,T]$, $P_s$ maps $\mathcal{B}_b(\Omega)$ 
			into $\mathcal{B}^s_b(\Omega)$;
			\item for all $s,t\in[0,T]$ with $t\geq s$, $P_s\circ P_t=P_s$.
		\end{itemize}
	\end{enumerate}
\end{definition}

The proposition below states a correspondence between path-dependent canonical classes and path-dependent systems of projectors. It was the object of
Corollary 3.1 of  \cite{paperMPv2}.
\begin{proposition}\label{EqProbaOp}
	The mapping 
	\begin{equation}
	(\mathbbm{P}^{s,\eta})_{(s,\eta)\in[0,T]\times \Omega}\longmapsto\left(\begin{array}{rcl}Z&\longmapsto& (\eta\mapsto\mathbbm{E}^{s,\eta}[Z])\\ \mathcal{B}_b(\Omega)&\longrightarrow& \mathcal{B}_b(\Omega)\end{array}\right)_{s\in[0,T]},
	\end{equation}
	is a bijection between the set of path-dependent system of probability measures and the set of path-dependent system of projectors.
\end{proposition}

\begin{definition}\label{ProbaOp}
	Two elements in correspondence through the previous bijection will be said to be \textbf{associated}. 
\end{definition}

\begin{notation} \label{N310}
	Let $(P_s)_{s\in[0,T]}$ be a path-dependent system of projectors, and
	$(\mathbbm{P}^{s,\eta})_{(s,\eta)\in[0,T]\times \Omega}$ the associated path-dependent system of probability measures.
	Then for any r.v. $Z \in \mathcal{L}^1(\mathbbm{P}^{s,\eta})$, $P_s[Z](\eta)$ will still denote the expectation of $Z$ under $\mathbbm{P}^{s,\eta}$. In other words we extend the linear form $Z\longmapsto P_s[Z](\eta)$ from $\mathcal{B}_b(\Omega)$ to $\mathcal{L}^1(\mathbbm{P}^{s,\eta})$.
	
	If $Z:=(Z^1,\cdots,Z^d)$ is an $\mathbbm{R}^d$-valued r.v., then for all $s$, $P_s[Z]$ (if well-defined) will denote the  $\mathbbm{R}^d$-valued r.v. $(P_s[Z^1],\cdots,P_s[Z^d])$.
\end{notation}

For the results of this section, we are given a progressive path-dependent canonical class $(\mathbbm{P}^{s,\eta})_{(s,\eta)\in[0,T]\times \Omega}$ satisfying Hypothesis \ref{HypClass}
and the corresponding path-dependent system of projectors $(P_s)_{s\in[0,T]}$.




The statement below comes from Corollary 3.3 of \cite{paperMPv2}.
\begin{proposition}\label{CoroTrivial} 
	For every $(s,\eta)\in\mathbbm{R}_+\times\Omega$ and $F\in\mathcal{F}_s$, $\mathbbm{P}^{s,\eta}(F)\in\{0,1\}$. In particular, an $\mathcal{F}^{s,\eta}_s$-measurable r.v. will be $\mathbbm{P}^{s,\eta}$-a.s. equal to a constant.
\end{proposition}

The last notions and results of this subsection are taken from 
Section 5.2 of \cite{paperMPv2}.

We consider a couple $(\mathcal{D}(A),A)$ verifying the following.
\begin{hypothesis}\label{HypDA}
	\begin{enumerate}\
		\item $\mathcal{D}(A)$ is a linear subspace of the space of    $\mathbbm{F}^o$-progressively measurable processes;
		\item $A$ is a linear mapping from $\mathcal{D}(A)$ into the space of  $\mathbbm{F}^o$-progressively measurable processes;
		\item for all $\Phi\in\mathcal{D}(A)$, $\omega\in\Omega$, $\int_0^T|A\Phi_r(\omega)|dr<+\infty$;
		\item for all $\Phi\in\mathcal{D}(A)$, $(s,\eta)\in[0,T]\times \Omega$ and 
		$t\in[s,T]$, we have  
		\\
		$\mathbbm{E}^{s,\eta}\left[\int_{s}^{t}|A(\Phi)_r|dr\right]<+\infty$ and $\mathbbm{E}^{s,\eta}[|\Phi_t|]<+\infty$.
	\end{enumerate}
\end{hypothesis}

\begin{definition}\label{MPop}
\	\begin{enumerate}
		\item
		$(\mathbbm{P}^{s,\eta})_{(s,\eta)\in[0,T]\times \Omega}$ will be said to solve the \textbf{martingale problem associated to }$(\mathcal{D}(A),A)$ if for every $(s,\eta)\in[0,T]\times\Omega$, 
		\begin{itemize}
			\item $\mathbbm{P}^{s,\eta}(\omega^s=\eta^s)=1$;
			\item $\Phi-\int_0^{\cdot}A(\Phi)_rdr$,  is on $[s,T]$ a $(\mathbbm{P}^{s,\eta},\mathbbm{F}^o)$-martingale
			for all $\Phi\in\mathcal{D}(A)$.
		\end{itemize} 
		\item The martingale problem associated to $(\mathcal{D}(A),A)$ will be said to be \textbf{well-posed} if for every  $(s,\eta)\in[0,T]\times\Omega$ there exists a unique $\mathbbm{P}^{s,\eta}$ verifying both items above.
	\end{enumerate}
\end{definition}
Inspired from the classical literature (see 13.28 in \cite{jacod}) we have introduced in \cite{paperMPv2} the following notion of a weak  generator.
\begin{definition}\label{WeakGen} 
	We say that $(\mathcal{D}(A),A)$ is a \textbf{weak generator} of a path-dependent system of projectors $(P_s)_{s\in[0,T]}$ if for all $\Phi\in\mathcal{D}(A)$, $(s,\eta)\in[0,T]\times \Omega$ and 
	$t\in[s,T]$, we have
	\begin{equation}
	P_s[\Phi_t](\eta)=\Phi_s(\eta)+\int_s^tP_s[A(\Phi)_r](\eta)dr.
	\end{equation}
\end{definition}

The proposition below was the object of Proposition 5.6 in \cite{paperMPv2}.
\begin{proposition}\label{MPopWellPosed}
	$(\mathcal{D}(A),A)$ is a weak generator of  $(P_s)_{s\in[0,T]}$ if and only if $(\mathbbm{P}^{s,\eta})_{(s,\eta)\in[0,T]\times \Omega}$ solves the martingale problem associated to $(\mathcal{D}(A),A)$.
	
	In particular, if $(\mathbbm{P}^{s,\eta})_{(s,\eta)\in[0,T]\times \Omega}$ solves the well-posed martingale problem associated to  $(\mathcal{D}(A),A)$ then $(P_s)_{s\in[0,T]}$ is the unique path-dependent system of projectors for which $(\mathcal{D}(A),A)$ is a weak generator.
\end{proposition}

In the setup of the last statement, one can therefore associate analytically  to $(\mathcal{D}(A),A)$ a unique path-dependent system of projectors $(P_s)_{s\in[0,T]}$ through Definition \ref{WeakGen}.

\section{Path-dependent canonical classes induced by Gaussian measures}\label{SGauss}

\begin{notation}
	Let $(E,\| \cdot \|)$ be a Banach space and $F$ be a linear subspace of $E$ then  its closure will be denoted $\overline{F}^{\| \cdot\|} $ or $\overline{F}^{E}$ when there can be no ambiguity concerning the chosen norm.
\end{notation}
In this section we will also adopt the conventions of Section \ref{S1}.
Most of the following definitions are taken from \cite{bogachevGauss} Chapter 2.2.
\begin{definition}\label{DefBasic}
Let $\mathbbm{P}$ be a {\bf Gaussian measure} on $(\Omega,\mathcal{F})$, i.e. a probability measure such that for any $n\in \mathbbm{N}^*$ and $l_1,\cdots,l_n\in \Omega^*$,  $(l_1,\cdots,l_n)$ has under $\mathbbm{P}$ the law of a Gaussian vector.
Let $L^2(\mathbbm{P})$ denote the corresponding space of square integrable random variables and assume  that  $\underset{t\in[0,T]}{\text{sup }}\|X_t\|\in L^p(\mathbbm{P})$ for all $p\in\mathbbm{N}$.

\begin{itemize}
	\item We define the covariance operator of $\mathbbm{P}$  $K:\Omega^*\longmapsto \Omega$ by $Kl:t\mapsto \mathbbm{E}[\omega(t)l(\omega)]$ for all $l\in\Omega^{*}$. 
	\item We denominate {\bf covariance function} of $\mathbbm{P}$ the (symmetric matrix valued) function $c:(s,t)\longmapsto \mathbbm{E}[X_sX_t]$, and {\bf mean function} of $\mathbbm{P}$ the function $m:t\longmapsto \mathbbm{E}[X_t]$. The $(i,j)$-th coordinate of $c$ will be denoted $c_{i,j}$.
	\item We say that  $\mathbbm{P}$ is {\bf of full support} if the smallest closed subset of $\Omega$ of measure $1$ is $\Omega$.

	\item We say that $X$ admits a {\bf representation} under $\mathbbm{P}$ if the following holds. There exists a function $k:[0,T]^2\longrightarrow \mathcal{M}_d(\mathbbm{R})$ such that for all $t\in [0,T]$, $k(t,\cdot)\in L^2([0,t])$ and taking value $0$ on $]t,T]$;    and an $\mathbbm{F}^o$-adapted $\mathbbm{R}^d$-valued Brownian motion $B:=(B^1,\cdots,B^d)$ defined on $(\Omega,\mathcal{F})$ such that
	\begin{equation}\label{Repk}
		X_t = \int_0^tk(t,r)dB_r,\quad \mathbbm{P}\text{ a.s.}
	\end{equation}
  In this case, $k$ is called the \textbf{kernel function} of $\mathbbm{P}$.
  \item We call {\bf Reproducing Kernel Hilbert Space} (RKHS) 
the Hilbert space of centered elements of 
$\overline{Span(\{X^i_r|i\leq d;r\in[0,T]\})}^{L^2(\mathbbm{P})}$ which we denote $H(\mathbbm{P})$, we also denote for all $t\in[0,T]$ the Hilbert subspace  $H^t(\mathbbm{P})$ of centered elements of $\overline{Span(\{X^i_r|i\leq d;r\in[0,t]\})}^{L^2(\mathbbm{P})}$.
	\item We call {\bf Cameron-Martin space}
 which we denote  $\mathcal{H}(\mathbbm{P})$, the space of functions $c_Y:t\longmapsto \mathbbm{E}[Y X_t]$ for $Y\in H(\mathbbm{P})$, we equip it with the scalar product defined by $(c_Y,c_Z)_{\mathcal{H}(\mathbbm{P})}:=\mathbbm{E}[YZ]$ for all $Y,Z\in H(\mathbbm{P})$, which makes it a Hilbert space. We also denote for all $t\in[0,T]$ the Hilbert subspace  $\mathcal{H}^t(\mathbbm{P})$ of functions $c_Y:t\longmapsto \mathbbm{E}[Y X_t]$ for $Y\in H^t(\mathbbm{P})$.
	\item We say that $(\mathcal{H}(\mathbbm{P}),\Omega,\mathbbm{P})$ is an 
{\bf abstract Wiener space} if $\mathcal{H}(\mathbbm{P})$ is a dense (for $\|\cdot\|_{\infty}$) subspace of $\Omega$.
\end{itemize}
\end{definition}
\begin{remark} \label{R35}
About Definition \ref{DefBasic}
we mention the following. 
\begin{itemize}
\item If $\mathbbm{P}$ is a Gaussian measure on $(\Omega,\mathcal{F})$ (see Definition 2.2.1 in \cite{bogachevGauss}), then the canonical process $X$ (is under $\mathbbm{P}$) a Gaussian process;

\item  $K$ indeed maps $\Omega^*$ into $\Omega$ because of Theorem 3.2.3 in \cite{bogachevGauss};
\item  our definition of $\mathcal{H}(\mathbbm{P})$ is not the one of \cite{bogachevGauss} Chapter 2.2, but is equivalent again by Theorem 3.2.3 ibidem, which also ensures  that elements of $\mathcal{H}(\mathbbm{P})$ belong to $\Omega$;
\item  $\mathbbm{P}$ is of full support if and only if $\mathcal{H}(\mathbbm{P})$ is dense in $\Omega$ for $\|\,\|_{\infty}$ (i.e. $(\mathcal{H}(\mathbbm{P}),\Omega,\mathbbm{P})$ is an abstract Wiener space), see Theorem 3.6.1 in \cite{bogachevGauss}.
\end{itemize}
\end{remark}

We consider a Gaussian probability measure $\mathbbm{P}$ on $(\Omega,\mathcal{F})$ verifying the following.
\begin{hypothesis}\label{HypGaussProba}
	\begin{enumerate}\
		\item $\underset{t\in[0,T]}{\text{sup }}\|X_t\|\in L^p(\mathbbm{P})$ for every $p \ge 1$;
		\item $\mathbbm{P}$ is of full support;
		\item  
$X$ admits a representation under $\mathbbm{P}$ with respect to some 
Brownian motion $B$, with a kernel $k$, see \eqref{Repk}.
		\item for all $t\in[0,T]$ and $h\in L^2([0,T])$ we have that $\int_0^{\cdot}k(\cdot,r)h(r)dr \equiv 0$ on $[0,t]$ implies that $h$ is Lebesgue a.e. equal to zero on $[0,t]$; 
		\item 
		\begin{equation}
		M_{op}:=\underset{i\leq d}{\text{max}}\underset{s\in[0,T]}{\text{sup }}\underset{r\in[0,s]}{\text{sup }}\frac{\underset{j\leq d}{\text{max}}\underset{t\in[0,T]}{\text{sup }}|c_{i,j}(r,t)|}{\underset{j'\leq d}{\text{max}}\underset{r'\in[0,s]}{\text{sup }}|c_{i,j'}(r,r')|}<+\infty.
		\end{equation}
	\end{enumerate}
\end{hypothesis}

\begin{remark}\label{RemHypGaussProba}
\begin{enumerate}\
\item Item 5. of Hypothesis \ref{HypGaussProba} is verified for example by the following processes.
	\begin{itemize}
		\item Stationary processes, see Corollary 5 in \cite{lagatta};
		\item the fractional Brownian motion of Hurst index 
$H\in]0,1[$, see the proof of Theorem 3.1 in \cite{Sottinen}.
	\end{itemize}
\item If items 3. and 4. in Hypothesis \ref{HypGaussProba} hold,
 then by Theorem 1.7 in 
\cite{HidaCanRep}, 
we know that for all $t\in[0,T]$, $H^t(\mathbbm{P})=\overline{Span(\{B^i_r|i\leq d;r\in[0,t]\})}^{L^2(\mathbbm{P})}$. 
\end{enumerate}
\end{remark}

\begin{notation}\label{Npis}
	\begin{itemize}\
		\item For all $s\in[0,T]$, we denote by $\Omega_s$ the Banach subspace of $\Omega$ constituted of paths $\omega$ constant after time $s$, i.e. such that $\omega=\omega^s$ and we denote by $\pi_s$ the continuous mapping 
		$\begin{array}{rcl}
		\Omega&\longrightarrow&\Omega_s\\
		\omega&\longmapsto& \omega^s.
		\end{array}$
		\item By a slight abuse of notation, we denote by $\Omega_s^{\perp}$ the Banach subspace $\Omega$ constituted of paths $\omega$ taking value $0$ on $[0,s]$, and by 
		$\pi_s^{\perp}$ the continuous mapping 
		$\begin{array}{rcl}
		\Omega&\longrightarrow&\Omega_s^{\perp}\\
		\omega&\longmapsto& \omega - \omega^s.
		\end{array}$
		
		\item $K:\Omega^{*}\longrightarrow\Omega$ denotes the covariance operator of $\mathbbm{P}$, see Definition \ref{DefBasic}.
		
		\item Let $k$ be the function appearing in \eqref{Repk}, then for all $i\leq d$, $k_i:[0,T]\times[0,T]\rightarrow\mathbbm{R}^d$ will denote its $i$-th column, and for all $i,j\leq d$, $k_{i,j}:[0,T]\times[0,T]\rightarrow\mathbbm{R}$ will denote its $(i,j)$-th coefficient.
	\end{itemize}
	
\end{notation}
In the proposition below for every $s \in [0,T], \eta \in \Omega_s$
we introduce a Gaussian probability measure $ P^{s,\eta}$
which represents the conditional law of \eqref{Repk}
given $\omega^s = \eta^s$.
\begin{proposition}\label{CondExp}\
\begin{enumerate}
\item $K\Omega^{*}$ is dense in $\Omega$.
\item For every $s\in[0,T]$, there exists a set of Gaussian probability measures $(\mathbbm{P}^{s,\eta})_{\eta\in\Omega_s}$ 
(with related expectations $(\mathbbm{E}^{s,\eta})_{\eta\in\Omega_s}$) 
and a continuous operator $m_s:\Omega_s\longrightarrow\Omega$ such that the following holds. 
	\begin{enumerate}
		\item For all $\eta\in\Omega_s$, $\mathbbm{P}^{s,\eta}(\omega^s=\eta)=1$;
		\item $\eta\longmapsto \mathbbm{P}^{s,\eta}$ is continuous;
		\item for every   $t\geq s$  and $F\in\mathcal{F}$,
		\begin{equation} \label{regcondexp}
		\mathbbm{P}(F|\mathcal{F}^o_s)(\eta)=\mathbbm{P}^{s,\eta}(F)\text{ for }\mathbbm{P}\text{ almost all }\eta;
		\end{equation}
		\item for all $t\in[0,T]$, $\mathbbm{E}^{s,\eta}[X_t]=m_s[\eta](t)$;
		\item for all $s$, $m_s$ has an operator norm inferior to $M_{op}$;
		\item $\pi_sK\pi_s^{*}\Omega_s^{*}$ is dense in $\Omega_s$ and that on $K\pi_s^{*}\Omega_s^{*}$, $m_s\circ\pi_s$coincides with the identity.
	\end{enumerate}
\end{enumerate}
\end{proposition}
\begin{proof}
The  statement 1. follows from Theorem 2.1 in \cite{lagatta} and the fact that by Hypothesis \ref{HypGaussProba} item 2., the support of $\mathbbm{P}$ is $\Omega$.

The  statement 2. follows  from  Theorem 2.1 and Theorem 2 
 and the statement and proof of Lemma 2.2 in \cite{lagatta} applied for fixed $s$ to the continuous linear operator $\pi_s$ between the Banach spaces $\Omega$ and $\Omega_s$.

\end{proof}

\begin{definition}\label{extension}
\begin{enumerate}
\item For every $s\in[0,T]$, we extend $\eta\mapsto m_s[\eta]$ and $\eta\mapsto\mathbbm{P}^{s,\eta}$ from $\Omega_s$ to $\Omega$ by setting for all $\eta\in\Omega$, $m_s[\eta]:=m_s[\eta^s]$ and $\mathbbm{P}^{s,\eta}:=\mathbbm{P}^{s,\eta^s}$. 
\item  $(s,\eta)\mapsto m_s[\eta]$ will be called the {\bf mean random field}.
\end{enumerate} 
\end{definition}
By continuity of $\pi_s$, we remark that for all $s$, $\eta\mapsto m_s[\eta]$, $\eta\longmapsto \mathbbm{P}^{s,\eta}$ remain continuous.
\begin{remark}\label{ms}
	The family of linear operators $(m_s)_{s\in[0,T]}$ 
constituting the mean random field
is crucial  in this paper.
\begin{enumerate}
\item First, that consitutes  a very useful analytical tool, since it has permit
 to \cite{lagatta} to prove that the regular conditional expectation of $\mathbbm{P}$ is
 continuous, see Proposition \ref{CondExp} item 2.b.
\item
Second,	it is also  central at the  probabilistic level of this paper.
 In particular $m_s[\eta]$ is the mean function of $\mathbbm{P}^{s,\eta}$,
 the mean random field will also allow us to construct driving martingales 
 for  our BSDEs, as we will see in item 3. of Proposition \ref{mTprop}.
\end{enumerate}
\end{remark}

 The proof of Propositions \ref{ClassGauss} and \ref{mcad}  
  below is postponed to the Appendix.
\begin{proposition}\label{ClassGauss}
$(\mathbbm{P}^{s,\eta})_{(s,\eta)\in[0,T]\times\Omega}$ is a path-dependent canonical class.
\end{proposition}

\begin{proposition}\label{mcad}
For all $\eta\in\Omega$, $s\longmapsto m_s[\eta]$ is right-continuous in $(\Omega,\|\cdot\|_{\infty})$, in particular, for all $t\in[0,T]$ $s\longmapsto m_s[\eta](t)$ is right-continuous.
\end{proposition}

\begin{corollary}\label{Coro_m}
$m:\begin{array}{rcl}
(s,\eta)&\longmapsto&m_s[\eta]\\
\,[0,T]\times\Omega&\longrightarrow&\Omega
\end{array}$ is $(\mathcal{P}ro^o,\mathcal{F})$-measurable.

Moreover, for all $t\in[0,T]$, $\begin{array}{rcl}
(s,\eta)&\longmapsto&m_s[\eta](t)\\
\,[0,T]\times\Omega&\longrightarrow&\mathbbm{R}^d
\end{array}$ is $\mathbbm{F}^o$-progressively measurable.
\end{corollary}
\begin{proof}
Let $T_0\in[0,T]$.
$m:[0,T_0]\times\Omega_{T_0}\mapsto\Omega$ is right-continuous
in $s \in [0,T_0]$ at fixed $\eta \in \Omega_{T_0}$ by Proposition \ref{mcad} and continuous in $\eta$ at fixed $s$ hence jointly measurable, see Theorem 15 in \cite{dellmeyer75} Chapter IV.

For all $t\leq T_0$ and $\omega\in\Omega$, we have $m_t(\omega)=m_t\circ\pi_t(\omega)=m_t\circ\pi_{T_0}(\omega)$, taking into account Definition \ref{extension}.

So $m:[0,T_0]\times\Omega\mapsto\Omega$ can be expressed composing $m:[0,T_0]\times\Omega_{T_0}\mapsto\Omega$ with  
$$ \begin{array}{rcl}
(t,\omega)&\mapsto&(t,\pi_{T_0}(\omega))\\
\,[0,T_0]\times\Omega&\rightarrow& [0,T_0]\times\Omega_{T_0}
\end{array},$$
 which is clearly $\left(\mathcal{B}([0,T_0])\otimes\mathcal{F}^o_{T_0},\mathcal{B}([0,T_0])\otimes\mathcal{B}(\Omega_{T_0})\right)$-measurable. By composition, 
 $m:[0,T_0]\times\Omega\mapsto\Omega$ is $\mathcal{B}([0,T_0])\otimes\mathcal{F}^o_{T_0}$-measurable.  
Since this holds for all $T_0$, the first statement is shown.

The second part of the statement follows composing $m_s(\eta)$ 
and $X_t$  which is  continuous hence measurable  for all $t$.
\end{proof}

\begin{remark}\label{RPs}
	For every $(s,\eta)$, $X - m_s[\eta]$ is under $\mathbbm{P}^{s,\eta}$ a
 mean-zero continuous Gaussian process whose 
  covariance function does not depend on $\eta$, see Theorem 2 in \cite{lagatta}.
  
  In particular, if $s$ is fixed and if we consider two paths $\eta,\eta'$ in $\Omega$, then $\mathbbm{P}^{s,\eta'}$ is the translation of $\mathbbm{P}^{s,\eta}$ by the vector $m_s[\eta]- m_s[\eta']=m_s[\eta-\eta']$.
	\end{remark}
\begin{notation}\label{Ps}

	For every $(s,\eta)$, we denote by $c^s$ the  covariance function of 
 $X - m_s[\eta]$ under $\mathbbm{P}^{s,\eta}$. 
 We denote by $\mathbbm{P}^s$
 the law of that process, which does not depend on $\eta$.
The expectation under $\mathbbm{P}^s$ will be denoted by $\mathbbm{E}^s$.
\end{notation}

\begin{proposition}\label{supX}
For every $(s,\eta)\in[0,T]\times\Omega$ and $p\in\mathbbm{N}^*$, $\underset{r\in[0,T]}{\text{sup }}\|X_r\|\in\mathcal{L}^p(\mathbbm{P}^{s,\eta})$.
\end{proposition}
\begin{proof}
We fix $s\in[0,T]$ and $p\in\mathbbm{N}^*$.
We start by noticing that for every $\eta\in\Omega$, 
\begin{equation}\label{EqSupX1}
\mathbbm{E}^{s,\eta}[\|\omega\|_{\infty}^p] = \mathbbm{E}^{s}[\|\omega+m_s[\eta]\|_{\infty}^p].
\end{equation}
 Then by triangle inequality for $\|\cdot\|_{\infty}$ and convexity of $x\mapsto x^p$, we can write 
\begin{equation}\label{EqSupX2}
\|\omega+m_s[\eta]\|_{\infty}^p\leq 2^{p-1}\|\omega\|_{\infty}^p+2^{p-1}\|m_s[\eta]\|_{\infty}^p.
\end{equation}
Since $\|\omega\|_{\infty}^p=\|-m_s[\eta]+(\omega+m_s[\eta])\|_{\infty}^p\leq 2^{p-1}\|m_s[\eta]\|_{\infty}^p+2^{p-1}\|\omega+m_s[\eta]\|_{\infty}^p$ then
\begin{equation}\label{EqSupX3}
\frac{1}{2^{p-1}}\|\omega\|_{\infty}^p-\|m_s[\eta]\|_{\infty}^p\leq \|\omega+m_s[\eta]\|_{\infty}^p.
\end{equation}
Taking the expectation $\mathbbm{E}^{s}$ in  \eqref{EqSupX2}, \eqref{EqSupX3} and taking \eqref{EqSupX1} into account yields
\begin{equation}
\frac{1}{2^{p-1}}\mathbbm{E}^{s}[\|\omega\|_{\infty}^p]-\left\|m_s[\eta]\right\|_{\infty}^p\leq \mathbbm{E}^{s,\eta}[\|\omega\|_{\infty}^p]\leq 2^{p-1}\mathbbm{E}^{s}[\|\omega\|_{\infty}^p] +2^{p-1}\left\|m_s[\eta]\right\|_{\infty}^p.
\end{equation}
So either $\mathbbm{E}^{s}[\|\omega\|_{\infty}^p]$ is finite and therefore $\mathbbm{E}^{s,\eta}[\|\omega\|_{\infty}^p]$ is finite for all $\eta$, or $\mathbbm{E}^{s}[\|\omega\|_{\infty}^p]$ is infinite and therefore $\mathbbm{E}^{s,\eta}[\|\omega\|_{\infty}^p]$ is infinite for all $\eta$. We now show that the second option is not possible in order to conclude. Indeed, by Remark \ref{C7Borel} we have
\begin{equation}
\begin{array}{rcl}
\mathbbm{E}\left[\|\omega\|_{\infty}^p\right]&=&\mathbbm{E}\left[\mathbbm{E}\left[\|\omega\|_{\infty}^p|\mathcal{F}^o_s\right](\eta)\right]\\
&=&\mathbbm{E}\left[\mathbbm{E}^{s,\eta}\left[\|\omega\|_{\infty}^p\right]\right],
\end{array}
\end{equation}
where we recall in particular that $\mathbbm{E}\left[\|\omega\|_{\infty}^p\right]<+\infty$ by Hypothesis \ref{HypGaussProba}, so $\mathbbm{E}^{s,\eta}[\|\omega\|_{\infty}^p$ is finite for $\mathbbm{P}$ almost all $\eta$ hence is not infinite for all $\eta$.

\end{proof}
The proposition below is proved in the Appendix. 
\begin{proposition}\label{FullSupp}
	For every $(s,\eta)\in[0,T]\times \Omega$, the topological support of $\mathbbm{P}^{s,\eta}$ is equal to $\eta^s +\Omega_s^{\perp}$ i.e. the set of paths coinciding with $\eta$ on $[0,s]$. 
\end{proposition}

We recall that the covariance functions $c^s$ have been defined at Notation 
 \ref{Ps}.
\begin{lemma}\label{Cs}
	For every $s,t,u\in[0,T]$, 
	\begin{equation}
	c^s(t,u)=\left\{\begin{array}{l}
	\int_s^{t\wedge u}k(t,r)k(u,r)dr\quad \text{if }s\leq t,u\\
	0\quad \text{otherwise.}
	\end{array}\right.
	\end{equation}
\end{lemma}
\begin{proof}
	We fix $t,u$. For every $\eta\in\Omega$, we have 
	\begin{equation}\label{Eq7}
	c^s(t,u)=\mathbbm{E}^{s,\eta}[X_t\otimes X_u]-\mathbbm{E}^{s,\eta}[X_t]\otimes\mathbbm{E}^{s,\eta}[X_u].
	\end{equation}
	Clearly if $t$ (resp. $u$) is inferior to $s$ then $X_t$ 
(resp. $X_u$) is for all $\eta$  $\mathbbm{P}^{s,\eta}$ a.s. deterministic, 
see 
 Proposition \ref{CondExp} 2. (a).
 This implies $c^s(t,u)=0$. Assume now that $s\leq t,u$.
	By Proposition \ref{CondExp} 2. (c) and
 \eqref{Eq7} we have $\mathbbm{P}$ a.s. that, 
	\begin{equation}
	\begin{array}{rcl}

	c^s(t,u)&=&\mathbbm{E}[ X_t\otimes X_u|\mathcal{F}^o_s]-\mathbbm{E}[X_t|\mathcal{F}^o_s]\otimes\mathbbm{E}[X_u|\mathcal{F}^o_s] \\
	&=&\mathbbm{E}\left[\left(\int_0^tk(t,r)dB_r\right)\otimes\left(\int_0^uk(u,r)dB_r\right)|\mathcal{F}^o_s\right]\\
	&&-\mathbbm{E}[\int_0^tk(t,r)dB_r|\mathcal{F}^o_s]\otimes\mathbbm{E}[\int_0^tk(t,r)dB_r|\mathcal{F}^o_s]\\
	&=&\left(\int_0^sk(t,r)dB_r\right)\otimes\left(\int_0^sk(t,r)dB_r\right) + \int_s^{t\wedge u}k(t,r)k(u,r)dr\\ &&-\left(\int_0^sk(t,r)dB_r\right)\otimes\left(\int_0^sk(t,r)dB_r\right)\\
	&=&\int_s^{t\wedge u}k(t,r)k(u,r)dr,
	\end{array}
	\end{equation}
	and the proof is complete.
\end{proof}
The proof of the proposition below is also located in the Appendix.
\begin{proposition}\label{Prog}
$(\mathbbm{P}^{s,\eta})_{(s,\eta)\in[0,T]\times\Omega}$ is progressive, see Definition \ref{DefCondSyst}.
\end{proposition}

\begin{notation}\label{CompletedBasis}
	For any $(s,\eta)\in[0,T]\times \Omega$ we will consider the  stochastic basis $\left(\Omega,\mathcal{F}^{s,\eta},\mathbbm{F}^{s,\eta}:=(\mathcal{F}^{s,\eta}_t)_{t\in[0,T]},\mathbbm{P}^{s,\eta}\right)$ where $\mathcal{F}^{s,\eta}$ (resp. $\mathcal{F}^{s,\eta}_t$ for all $t$) is $\mathcal{F}$ (resp. $\mathcal{F}_t$) augmented with  the $\mathbbm{P}^{s,\eta}$ negligible sets. $\mathbbm{P}^{s,\eta}$ is extended to $\mathcal{F}^{s,\eta}$.
\end{notation}
We remark that, for any $(s,\eta)\in[0,T]\times \Omega$, $\left(\Omega,\mathcal{F}^{s,\eta},\mathbbm{F}^{s,\eta},\mathbbm{P}^{s,\eta}\right)$ is a stochastic basis fulfilling the usual conditions, see 1.4 in \cite{jacod}
Chapter I. 

\begin{proposition}\label{FiltRichtCont}
Let $X^{s,\eta}$ denote the process $X-m_s[\eta]$. Then there exists an $\mathbbm{R}^d$-valued $(\mathbbm{P}^{s,\eta},\mathbbm{F}^{s,\eta})$-Brownian motion
 $B^{s,\eta}$ starting in $s$ such that for all $t\geq s$, $X^{s,\eta}_t=\int_s^t k(t,r)dB^{s,\eta}_r$ $\mathbbm{P}^{s,\eta}$ a.s.
Moreover, for all $t\geq s$,
then $\mathcal{F}^{s,\eta}_t$, coincides with
 $\mathcal{F}^o_t$ augmented with $\mathbbm{P}^{s,\eta}$-null sets.

\end{proposition}

\begin{corollary}\label{CoroHyp} 
$(\mathbbm{P}^{s,\eta})_{(s,\eta)\in[0,T]\times\Omega}$ verifies Hypothesis \ref{HypClass}.
\end{corollary}
\begin{proof}
By Proposition \ref{ClassGauss} $(\mathbbm{P}^{s,\eta})$
is a path-dependent canonical class, in particular
 \eqref{DE13} holds. Taking into account Notation \ref{CompletedBasis}
the result follows by Proposition  \ref{FiltRichtCont}.

\end{proof}

We now conclude this section by extending previous results to the case with drift. In Proposition  \ref{ProbaDrift} below, $\mathbbm{Q}^{s,\eta}$ will model the conditional 
 law of $\beta+\int_0^{\cdot}k(\cdot,r)dB_r$ given $\omega^s = \eta^s$
 where $\beta$
is a given path in $\Omega$.

\begin{proposition}\label{ProbaDrift}
	Let $\beta\in\Omega$ and define for all $(s,\eta)$, and $F\in\mathcal{F}$: $ \mathbbm{Q}^{s,\eta}(F):=\mathbbm{P}^{s,\eta-\beta^s}(F-\beta)$.
	
	Then, 
the following holds.
\begin{enumerate}
\item $\left(\mathbbm{Q}^{s,\eta}\right)_{(s,\eta)\in[0,T]\times\Omega}$ is a progressive path-dependent canonical class satisfying Hypothesis \ref{HypClass}.
\item	
	For all $(s,\eta)$, $\mathbbm{Q}^{s,\eta}$ is a Gaussian measure with mean function $\beta+m_s[\eta-\beta]$  and covariance function $c^s$, see Definition \ref{DefBasic}.
\item	 
For all $(s,\eta)$,  under $\mathbbm{Q}^{s,\eta}$, 
there exists
 a Brownian motion $\tilde{B}^{s,\eta}$ such that on $[s,T]$, $X$ is indistinguishable from $m_s[\eta-\beta]+\beta + \int_s^{\cdot}k(\cdot,r)d\tilde{B}^{s,\eta}_r$.
\end{enumerate}
\end{proposition}

\begin{proof}
	We start with the first statement. 
The progressivity property follows by the one of $(\mathbbm{P}^{s,\eta})_{s,\eta}$.
Since $(\mathbbm{P}^{s,\eta})_{s,\eta}$ is a  path-dependent canonical class
items 1. and 2. of Definition \ref{DefCondSyst} are clearly verified, so we only have to show that \eqref{DE14} holds. 
	We fix $(s,\eta)$, $t\geq s$, $F\in\mathcal{F}$ and we show that
	\begin{equation}\label{EqPtilde}
		\mathbbm{Q}^{s,\eta}(F|\mathcal{F}_t)=\mathbbm{Q}^{t,\omega}(F),\quad \mathbbm{Q}^{s,\eta},\text{ a.s.}
	\end{equation}
	
	Let $G\in\mathcal{F}_t$. We have
	\begin{equation}\label{EqPtilde2}
	\begin{array}{rcl}
			\mathbbm{E}^{\mathbbm{Q}^{s,\eta}}[\mathds{1}_F(\omega)\mathds{1}_G(\omega)]&=&
			\mathbbm{E}^{s,\eta-\beta^s}[\mathds{1}_{F-\beta}(\omega)\mathds{1}_{G-\beta}(\omega)]\\
			&=&\mathbbm{E}^{s,\eta-\beta^s}[\mathbbm{E}^{s,\eta-\beta^s}[\mathds{1}_{F-\beta}|\mathcal{F}_t](\omega)\mathds{1}_{G-\beta}(\omega)]\\
			&=&\mathbbm{E}^{s,\eta-\beta^s}[\mathbbm{P}^{t,\omega}(F-\beta)\mathds{1}_{G-\beta}(\omega)]\\
			&=&\mathbbm{E}^{s,\eta-\beta^s}[\mathbbm{Q}^{t,\omega+\beta^t}(F)\mathds{1}_{G-\beta}(\omega)]\\
			&=&\mathbbm{E}^{\mathbbm{Q}^{s,\eta}}[\mathbbm{Q}^{t,\omega}(F)\mathds{1}_{G}(\omega)],
	\end{array}
	\end{equation}
where the second equality  holds because $G-\beta\in\mathcal{F}_t$; the third 
equality because $(\mathbbm{P}^{s,\eta})_{(s,\eta)\in[0,T]\times\Omega}$ verifies Hypothesis \ref{HypClass} (see Corollary \ref{CoroHyp}) and the last two
equalities by definition of the $\mathbbm{Q}^{s,\eta}$.

By definition of conditional expectation, the fact that \eqref{EqPtilde2} holds for all $G\in\mathcal{F}_t$ implies \eqref{EqPtilde}.

Concerning the second statement, we fix $(s,\eta)$. 
 $\mathbbm{Q}^{s,\eta}$ is the translation of $\mathbbm{P}^{s,\eta-\beta^s}$ in the direction $\beta$,
 so it is a Gaussian measure with same covariance function $c^s$ and 
with mean function the mean function (see Definition \ref{DefBasic}) of $\mathbbm{P}^{s,\eta-\beta^s}$, translated of $\beta$, meaning $m_s[\eta-\beta]+\beta$.

Finally the third statement is a consequence of Proposition \ref{FiltRichtCont} and of item 2.
\end{proof}

\section{BSDEs with Gaussian forward process and decoupled mild solutions of path-dependent PDEs}
\label{S4}

\subsection{General considerations.}\label{S41}

This section is the main part of the paper. Its aim is to introduce formally equation $PDPDE(f,\xi)$ introduced in the introduction (see \eqref{C7IntroPDE}), its coefficients, the operators that it involves, and to prove existence and uniqueness of what we call a \textit{decoupled mild solution}. We will make use of the probabilistic framework and results obtained in the previous section.

We are now given a Gaussian measure $\mathbbm{P}$ satisfying Hypothesis \ref{HypGaussProba} and the corresponding  path-dependent canonical class $(\mathbbm{P}^{s,\eta})_{(s,\eta)\in[0,T]\times\Omega}$, see Proposition \ref{ClassGauss}.

We fix a  function $b:[0,T]\times[0,T]\rightarrow\mathbbm{R}^d$ and we assume for the remainder of the paper that $b,k$ verify the following.

\begin{hypothesis}\label{Hypbk}\
	\begin{itemize}
		\item $b,k$ are bounded Borel;
		\item for all $s\in [0,T]$, $b(s,\cdot)$ and $k(s,\cdot)$  are equal to $0$ on $[0,s[$ and continuous on $[s,T]$ where they admit a bounded right-derivative;
		\item $t\mapsto\int_0^tb(t,r)dr$ is continuous.
	\end{itemize}
\end{hypothesis}

\begin{definition}  \label{D42}
	We set $\beta:t\mapsto\int_0^tb(t,r)dr$ and define $(\mathbbm{Q}^{s,\eta})_{(s,\eta)\in[0,T]\times\Omega}$ as in Proposition \ref{ProbaDrift} with this specific choice of $\beta$.
\end{definition}
\begin{notation}
	In this section, the continuous operator $m_s[\cdot-\beta]$ will be denoted $m_s$ and $\mathbbm{E}^{s,\eta}$ will denote the expectation with respect to $\mathbbm{Q}^{s,\eta}$ and not $\mathbbm{P}^{s,\eta}$ any more.
\end{notation}

We recall that by Proposition \ref{ProbaDrift}, we have the following.
\begin{remark}\label{ResumeQ}\
	\begin{itemize}
		\item $(\mathbbm{Q}^{s,\eta})_{(s,\eta)\in[0,T]\times\Omega}$ defines a progressive path-dependent canonical class verifying Hypothesis \ref{HypClass};
		\item for every $(s,\eta)$, $\mathbbm{Q}^{s,\eta}$ is the Gaussian measure on $(\Omega,\mathcal{F})$ of covariance function $c^s$ and mean function $m_s[\eta]+\beta$, see Definition \ref{DefBasic};	
	\item for every $(s,\eta)$, there exists a $\mathbbm{Q}^{s,\eta}$-Brownian motion $B^{s,\eta}$ such that under $\mathbbm{Q}^{s,\eta}$ we have 
		\begin{equation}
			X = m_s[\eta]+\int_0^{\cdot}b(\cdot,r)dr + \int_s^{\cdot}k(\cdot,r)dB^{s,\eta}_r,
		\end{equation}
		on $[s,T]$.
	\end{itemize}
\end{remark}

\subsection{Differential operators involved in the path-dependent PDE}\label{S42}

\begin{notation} \label{N45}
  From now on, $(P_s)_{s\in[0,T]}$ denotes the path-dependent system of
  projectors associated (in the sense of Definition \ref{ProbaOp}) to $(\mathbbm{Q}^{s,\eta})_{(s,\eta)\in[0,T]\times\Omega}$.
\end{notation}

Our aim now is to provide a weak generator $(\mathcal{D}(A),A)$ of $(P_s)_{s\in[0,T]}$, see Definition \ref{WeakGen}.
\\
\\
The following definitions are adapted from \cite{ViensZh} Section 3.1.

\begin{definition}\label{DefViens}
We denote by $\bar{\Omega}:=\mathbbm{D}([0,T],\mathbbm{R})$ the \textbf{Skorohod space} of cadlag functions from $[0,T]$ to $\mathbbm{R}$. $\bar{\mathcal{F}}$ (resp. $\bar{\mathbbm{F}}^o$) will denote the corresponding Borel $\sigma$-field (resp. initial filtration, see Notation \ref{canonicalspace}).

Let $t\in[0,T]$.
 $\bar{\Omega}_t$ will denote the set of elements of $\bar{\Omega}$ equal to $0$ on $[0,t[$ and continuous on $[t,T]$.
For $\omega,\eta\in\bar{\Omega}$. $\eta\otimes_t\omega$ will denote $\eta\mathds{1}_{[0,t[}+\omega\mathds{1}_{[t,T[}$.

 $\bar{\Lambda}$ will denote the set of $(t,\omega)\in[0,T]\times\bar{\Omega}$ such that $\omega$ is continuous after time $t$. We equip $\bar{\Lambda}$ with the distance defined by $d((s,\eta),(t,\omega))=|t-s|+\|\omega-\eta\|_{\infty}$.
$\mathcal{C}^0(\bar{\Lambda})$ will denote the set of real-valued  functions  on $\bar{\Lambda}$, continuous with respect to $d$.

We fix $\Phi\in\mathcal{C}^0(\bar{\Lambda})$. For $(t,\omega)\in\bar{\Lambda}$, $D\Phi_{t}(\omega)$ will denote 
$\underset{\epsilon \rightarrow  0+}{\text{lim }} \frac{1}{\epsilon}
\left(\Phi_{t+\epsilon}(\omega)-\Phi_{t}(\omega)\right)$ if this limit exists. \\

Let $\eta\in\bar{\Omega}_t$, $\nabla_{\eta}\Phi_{t}(\omega)$ will denote $\underset{\epsilon\rightarrow 0}{\text{lim }} \frac{1}{\epsilon}\left(\Phi_{t}(\omega+\epsilon\eta)-\Phi_{t}(\omega)\right)$ if this limit exists. We define $\nabla^2_{\eta}\Phi_{t}(\omega)$ similarly for $\eta\in\bar{\Omega}_t$. 

We say that $\Phi$ \textbf{has polynomial growth} if there exists $C>0$, $p\geq 1$ such that $|\Phi_{t}(\omega)|\leq C(1+\|\omega\|_{\infty}^p)$ for all $(t,\omega)\in\bar{\Lambda}$.

Concerning gradient processes, we will say that $\nabla\Phi$ has polynomial growth if there exists $C>0$, $p\geq 1$ such that $|\nabla_{\eta}\Phi_{t}(\omega)|\leq C(1+\|\omega\|_{\infty}^p)$ for all $(t,\omega)\in\bar{\Lambda}$ and $\eta\in\bar{\Omega}_t$.

We say that $\nabla\Phi$ is continuous if for all $\eta\in\bar{\Omega}_t$, $(t,\omega)\mapsto\nabla_{\eta}\Phi_{t}(\omega)\in\mathcal{C}^0(\bar{\Lambda})$.

Finally we define $\mathcal{C}_+^{1,2}(\bar{\Lambda})$ the set of elements $\Phi\in\mathcal{C}^0(\bar{\Lambda})$ verifying the following hypothesis.
\begin{itemize}
\item $D\Phi,\nabla\Phi,\nabla^2\Phi$ exist and are continuous;
\item $\Phi,D\Phi,\nabla\Phi,\nabla^2\Phi$ have polynomial growth;
\item there exists $p\geq 1$ and a bounded modulus of continuity $\rho$  such that for all $(t,\omega),(t,\omega')\in\bar{\Lambda}$ and $\eta\in\bar{\Omega}_t$, $$|\nabla^2_{\eta}\Phi_{t}(\omega)-\nabla^2_{\eta}\Phi_{t}(\omega')|\leq (1+\|\omega\|_{\infty}^p+\|\omega'\|_{\infty}^p)\rho(\|\omega-\omega'\|_{\infty}).$$
\end{itemize} 
\end{definition}

In the sequel, given  $\tilde{\Phi}\in\mathcal{C}^0(\bar{\Lambda})$,
we will denote
 \begin{equation}  \label{EPhiTilde}
\Phi:(t,\omega) \longmapsto \tilde{\Phi}_{t}(m_t[\omega]).
\end{equation}


\begin{lemma}\label{linkderiv}
Let $\tilde{\Phi}\in\mathcal{C}^0(\bar{\Lambda})$.
Then the following holds.
\begin{enumerate}
\item   $\Phi$ is $\mathbbm{F}^o$-progressively measurable. 

\item If moreover $\tilde{\Phi}\in\mathcal{C}_+^{1,2}(\bar{\Lambda})$, 
$(t,\omega)\mapsto D\tilde{\Phi}_{t}(m_t[\omega])$, $(t,\omega)\mapsto \nabla_{b(\cdot,t)}\tilde{\Phi}_{t}(m_t[\omega])$; $(t,\omega)\mapsto \nabla^2_{k_i(\cdot,t)}\tilde{\Phi}_{t}(m_t[\omega]), i\leq d$ 
 are also $\mathbbm{F}^o$-progressively measurable.
\end{enumerate}
\end{lemma}
\begin{proof}
\begin{enumerate}
\item
We fix $\tilde{\Phi}\in\mathcal{C}^0(\bar{\Lambda})$. By Corollary \ref{Coro_m}, $(t,\omega)\mapsto m_t[\omega]$ is
 $\mathbbm{F}^o$-progressively measurable, so  that 
$\Psi:\begin{array}{ccl}
(t,\omega)&\longmapsto& (t,m_t[\omega])\\
\,[0,T]\times\Omega&\longrightarrow &[0,T]\times \Omega 
\end{array}$ is \\
$(\mathcal{P}ro^o,\mathcal{B}([0,T])\otimes\mathcal{F})$-measurable. 
We remark that $[0,T]\times\Omega$
is a subset of $\bar \Lambda$.
The restriction of $\tilde{\Phi}$ to  $[0,T]\times\Omega$
of $\bar \Lambda$
  is continuous for the usual topology hence $(\mathcal{B}([0,T])\otimes\mathcal{F},\mathcal{B}(\mathbbm{R}))$-measurable. 
By composition,
$\Phi=\tilde{\Phi}\circ \Psi$ is  $(\mathcal{P}ro^o,\mathcal{B}(\mathbbm{R}))$-measurable.
\item
We now discuss the statement 2. Since $D\tilde{\Phi} \in \mathcal{C}^0(\bar{\Lambda})$, by the  statement 1. of the lemma that $(t,\omega)\mapsto D\tilde{\Phi}_{t}(m_t[\omega])$ is progressively measurable. We will now show that the same holds for the first order space derivative $(t,\omega)\mapsto \nabla_{b(\cdot,t)}\tilde{\Phi}_{t}(m_t[\omega])$. 

Since $\tilde{\Phi}$ is in $\mathcal{C}_+^{1,2}(\bar{\Lambda})$, then by definition, for all $\eta$, $(t,\omega)\mapsto \nabla_{\eta}\tilde{\Phi}_t(\omega)$ is continuous. On the other hand, it is clear that for all $(t,\omega)$, $\eta\mapsto \nabla_{\eta}\tilde{\Phi}_t(\omega)$ is measurable as the limit of measurable mappings. So $(\eta,t,\omega)\mapsto \nabla_{\eta}\tilde{\Phi}_t(\omega)$ is jointly measurable, see Lemma 4.51 in \cite{aliprantis}. Since $b$ is Borel, then $t\mapsto b^t$ is Borel from $([0,T],\mathcal{B}([0,T]))$ into $(\bar{\Omega},\bar{\mathcal{F}})$. By composition, we get that $(t,\omega)\mapsto \nabla_{b(\cdot,t)}\tilde{\Phi}_t(\omega)$ is measurable. We can now conclude as for the first statement by composing with $(t,\omega)\mapsto (t,m_t(\omega))$.

Finally, similar arguments allow to  show the progressive measurability
of the second order space derivatives $(t,\omega)\mapsto \nabla^2_{k_i(\cdot,t)}\tilde{\Phi}_{t}(m_t[\omega]), i\leq d$ .
\end{enumerate}
\end{proof}

%

\begin{definition}\label{DefDA}  
We denote $\mathcal{D}(\tilde{A})$ to be the set of mappings $\tilde{\Phi}\in\mathcal{C}_+^{1,2}(\bar{\Lambda})$  such that $(t,\omega)\mapsto D\tilde{\Phi}_{t}(m_t[\omega])$, $(t,\omega)\mapsto \nabla_{b(\cdot,t)}\tilde{\Phi}_{t}(m_t[\omega])$; $(t,\omega)\mapsto \nabla_{k_i(\cdot,t)}\tilde{\Phi}_{t}(m_t[\omega]),i\leq d$, $(t,\omega)\mapsto \nabla^2_{k_i(\cdot,t)}\tilde{\Phi}_{t}(m_t[\omega]),i\leq d$ 
  have polynomial growth.
On that space we define the linear operator $\tilde{A}$ by setting,
 for all  $\tilde{\Phi}\in  \mathcal{D}(\tilde{A})$ and $t\in[0,T],$
 $$\tilde{A}\tilde{\Phi}_t:= D\tilde{\Phi}_t+\nabla_{b(\cdot,t)}\tilde{\Phi}_{t}+\frac{1}{2}\underset{i\leq d}{\sum}\nabla^2_{k_i(\cdot,t)}\tilde{\Phi}_{t}.$$
  
We then denote $\mathcal{D}(A)$ to be the set of processes $\Phi:(t,\omega)\longmapsto\tilde{\Phi}_{t}(m_t[\omega])$ where $\tilde{\Phi}\in  \mathcal{D}(\tilde{A})$, and $A$ to be the linear operator defined for all $\Phi:(t,\omega)\longmapsto\tilde{\Phi}_{t}(m_t[\omega])\in \mathcal{D}(A)$ by

\begin{equation}
A\Phi_t(\omega) := \tilde{A}\tilde{\Phi}_t(m_t(\omega)), \forall (t,\omega).
\end{equation}
\end{definition}

\begin{remark} \label{R415}
	$\mathcal{C}_+^{1,2}(\bar{\Lambda})$, $\mathcal{D}(\tilde{A})$ and $\mathcal{D}(A)$ are linear algebras.
\end{remark}

\begin{proposition} \label{P413}
$(\mathcal{D}(A),A)$ introduced in previous Definition \ref{DefDA} fulfills Hypothesis \ref{HypDA}. 
\end{proposition}
\begin{proof}
Items 1. and 2. of 
 Hypothesis \ref{HypDA}
 are fulfilled thanks to Lemma \ref{linkderiv}; items 3. and 4. follow from polynomial growth of $\Phi,A\Phi$ for all $\Phi\in\mathcal{D}(A)$ and the fact that for all $s,\eta$, $p\geq 1$,
 $\underset{t\in[0,T]}{\text{sup }}|X_t|\in L^p(\mathbbm{Q}^{s,\eta})$, see Proposition \ref{supX} and the fact that $\mathbbm{Q}^{s,\eta}$ is a translation of $\mathbbm{P}^s$.
\end{proof}

The next step  consists in proving that
 $(\mathbbm{Q}^{s,\eta})_{(s,\eta)\in[0,T]\times\Omega}$ solves the martingale problem associated to 
$(\mathcal{D}(A),A)$, see Definition \ref{MPop}. Indeed by Remark \ref{ResumeQ}, 
 for all $(s,\eta)$, under $\mathbbm{Q}^{s,\eta}$,
 the process $\tilde{X}^{s,\eta}:=X-m_s[\eta]$  
 solves the Volterra SDE
 \begin{equation}
 \tilde{X}^{s,\eta}_t=\int_s^tb(t,r)dr+\int_s^tk(t,r) dB^{s,\eta}_r,\quad t\in[s,T].
 \end{equation}
 In this framework Theorem 3.9 in \cite{ViensZh} implies the following
 chain rule formula.

 \begin{theorem}\label{ThViens}
For every $(s,\eta)\in[0,T]\times\Omega$ and $\tilde{\Phi}\in\mathcal{C}_+^{1,2}(\bar{\Lambda})$ we have
the following. For all $t\geq s$
\begin{equation}\label{EqThViens}
\begin{array}{rl}
\tilde{\Phi}_{t}(m_t[\omega]-m_s[\eta])=&\tilde{\Phi}_{s}(0)+\int_s^tD\tilde{\Phi}_{r}(m_r[\omega]-m_s[\eta])dr\\
&+\int_s^t \nabla_{b(\cdot,r)}\tilde{\Phi}_{r}(m_r[\omega]-m_s[\eta])dr\\
&+\frac{1}{2}\underset{i\leq d}{\sum}\int_s^t\nabla^2_{k_i(\cdot,r)}\tilde{\Phi}_{r}(m_r[\omega]-m_s[\eta]) dr\\
&+\underset{i\leq d}{\sum}\int_s^t\nabla_{k_i(\cdot,r)}\tilde{\Phi}_{r}(m_r[\omega]-m_s[\eta])dB^{i,s,\eta}_r,\quad \mathbbm{Q}^{s,\eta}\text{ a.s.}
\end{array}
\end{equation}

\end{theorem}

\begin{proposition}\label{PropWeakGen}
We suppose the validity of Hypotheses \ref{HypGaussProba} and \ref{Hypbk}. 
Then \\
$(\mathbbm{Q}^{s,\eta})_{(s,\eta)\in[0,T]\times\Omega}$ solves the martingale problem associated to $(\mathcal{D}(A),A)$, see Definition \ref{MPop}. Moreover, $(\mathcal{D}(A),A)$ is a weak generator of $(P_s)_{s\in[0,T]}$, see Definition \ref{WeakGen}.
\end{proposition}
\begin{proof}
We fix $(s,\eta)$. The first item of Definition \ref{MPop} holds by construction of $(\mathbbm{Q}^{s,\eta})_{(s,\eta)\in[0,T]\times\Omega}$, see Proposition \ref{CondExp} 2. 

We now fix $\Phi:(t,\omega)\longmapsto\tilde{\Phi}_{t}(m_t[\omega])\in\mathcal{D}(A)$ with $\tilde{\Phi}\in\mathcal{C}_+^{1,2}(\bar{\Lambda})$. It is not hard to see that $\tilde{\Phi}^{s,\eta}:(t,\omega)\mapsto\tilde{\Phi}_{t}(\omega+m_s[\eta])$ also belongs to $\mathcal{C}_+^{1,2}(\bar{\Lambda})$ with
\begin{itemize}
\item  $D\tilde{\Phi}^{s,\eta}_{t}(\omega)=D\tilde{\Phi}_{t}(\omega+m_s[\eta])$,
\item  $\nabla_{b(\cdot,t)}\tilde{\Phi}^{s,\eta}_{t}(\omega)=\nabla_{b(\cdot,t)}\tilde{\Phi}_{t}(\omega+m_s[\eta])$,
\item $\nabla^2_{k_i(\cdot,t)}\tilde{\Phi}^{s,\eta}_{t}(\omega)=\nabla^2_{k_i(\cdot,t)}\tilde{\Phi}_{t}(\omega+m_s[\eta]),i\leq d$,
\end{itemize}
for all $(t,\omega)$.
Applying  Theorem \ref{ThViens} to $\tilde{\Phi}^{s,\eta}$, we obtain
\begin{equation}\label{DecompoPhi}
\begin{array}{rl}
\tilde{\Phi}_{t}(m_t[\cdot])=&\tilde{\Phi}_{s}(m_s[\eta])+\int_s^t\left(D\tilde{\Phi}_{r}+\nabla_{b(\cdot,r)}\tilde{\Phi}_{r}+\frac{1}{2}\underset{i\leq d}{\sum}\nabla^2_{k_i(\cdot,r)}\tilde{\Phi}_{r}\right)(m_r)dr\\
&+\underset{i\leq d}{\sum}\int_s^t\nabla_{k_i(\cdot,r)}\tilde{\Phi}_{r}(m_r)dB^{i,s,\eta}_r,\quad t\in[s,T], 
\end{array}
\end{equation}
in the sense of $\mathbbm{Q}^{s,\eta}$-indistinguishability.
Therefore, by definition of $\Phi$  and
 $A$ in Definition \ref{DefDA}, 
$\Phi-\int_0^{\cdot}A\Phi_rdr$ is on $[s,T]$ a $(\mathbbm{Q}^{s,\eta},\mathbbm{F}^o)$-local martingale, since it is indistinguishable, from $\Phi_s(\eta) +\underset{i\leq d}{\sum} \int_s^{\cdot}\nabla_{k_i(\cdot,r)}\tilde{\Phi}_{r}(m_r)  dB^{i,s,\eta}_r$.

Since for all $i$, $(t,\omega)\mapsto \nabla_{k_i(\cdot,r)}\tilde{\Phi}_{r}(m_r[\omega])$
 is assumed to have polynomial growth and $\underset{t\in[s,T]}{\text{sup }}X_t\in L^2(\mathbbm{Q}^{s,\eta})$ then, for all $i$, 
$\underset{t\in[s,T]}{\text{sup }}|\nabla_{k_i(\cdot,r)}\tilde{\Phi}_{r}(m_r)|\in L^2(\mathbbm{Q}^{s,\eta})$,
 therefore finally $\underset{i\leq d}{\sum}\int_s^{\cdot}\nabla_{k_i(\cdot,r)}\tilde{\Phi}_{r}(m_r) dB^{i,s,\eta}_r$ is a martingale. So $\Phi-\int_0^{\cdot}A\Phi_rdr$ is on $[s,T]$ a $(\mathbbm{Q}^{s,\eta},\mathbbm{F}^o)$-martingale. Since this holds for any $(s,\eta)$ and $\Phi\in\mathcal{D}(A)$ then  $(\mathbbm{Q}^{s,\eta})_{(s,\eta)\in[0,T]\times\Omega}$ solves the martingale problem associated to $(\mathcal{D}(A),A)$.

The second part of statement follows by Proposition \ref{MPopWellPosed}.
\end{proof}

\begin{notation}\label{NotMPhi}
For every $\tilde \Phi \in {\mathcal D}(\tilde A)$
and $(s,\eta)\in[0,T]\times \Omega$, we denote $M[\Phi]^{s,\eta}$  the continuous  $\mathbbm{Q}^{s,\eta}$-martingale 
$\Phi-\Phi_s(\eta)-\int_s^{\cdot}A\Phi_rdr,$ indexed by $[s,T]$.
\end{notation}

A direct consequence of Lemma 3.14 in \cite{paperPathDep}, 
 taking into account  Notation \ref{NotMPhi}, 
is the following.
\begin{corollary}\label{CoroBracketGamma}
Let $\tilde \Phi, \tilde \Psi \in \mathcal{D}(A).$
 Then for all $(s,\eta)$
\begin{equation}
\langle M[\Phi]^{s,\eta},M[\Psi]^{s,\eta}\rangle = \underset{i\leq d}{\sum}\int_s^{\cdot}\nabla_{k_i(\cdot,r)}\tilde{\Phi}_{r}(m_r)\nabla_{k_i(\cdot,r)}\tilde{\Psi}_{r}(m_r)dr,
\end{equation}
with respect to ${\mathbb Q}^{s,\eta}$.

\end{corollary}

The following bilinear operator was introduced  in \cite{paperPathDep} 
in a general path-dependent framework.

\begin{notation}\label{C7NotGamma}
	Let $\Phi,\Psi\in\mathcal{D}(A)$. We denote by $\Gamma(\Phi,\Psi)$ the process 
$A(\Phi\Psi)-\Phi A(\Psi)-\Psi A(\Phi)$. If $\Phi$ or $\Psi$ is multidimensional, then 
we define 
 $\Gamma(\Phi,\Psi)$ as a vector or matrix, coordinate by coordinate.
		$\Gamma(\Phi,\Phi)$ will be denoted $\Gamma(\Phi)$.
\end{notation}
$\Gamma$ can be interpreted as a path-dependent extension of the concept of carré du champ operator in the theory of Markov processes.
\begin{proposition}\label{ExprGamma}
 For every $\tilde{\Phi}, \tilde \Psi \in\mathcal{D}(\tilde A)$, we have 
$$\Gamma(\Phi,\Psi)_t(\omega)=\underset{i\leq d}{\sum}\nabla_{k_i(\cdot,t)}\tilde{\Phi}_{t}\nabla_{k_i(\cdot,t)}\tilde{\Phi}_{t}(m_t[\omega]),  \forall
         t,\omega.$$

\end{proposition}
\begin{proof}
This directly follows from the fact that $D$ and for $\zeta$, $\nabla_{\zeta}$  verify the usual product rules.
\end{proof}

\subsection{Construction  of the driving martingale for the BSDE}\label{S43}

\begin{notation}\label{L2uni}	
	We will indicate by $dt\otimes d\mathbbm{Q}^{s,\eta}$ the measure on $\mathcal{B}([0,T])\otimes\mathcal{F}$ defined by $dt\otimes d\mathbbm{Q}^{s,\eta}(C)=\mathbbm{E}^{s,\eta}\left[\int_s^T\mathds{1}_C(r,\omega)dr\right]$, and by $\mathcal{L}^2(dt\otimes d\mathbbm{Q}^{s,\eta})$  the space of  $\mathbbm{F}^{s,\eta}$-progressively measurable processes $Y$ such that $\mathbbm{E}^{s,\eta}\left[\int_s^{T} |Y_r|^2dr\right]< \infty$. 
	
	$\mathcal{H}^2_0(\mathbbm{Q}^{s,\eta})$ will denote the space of $(\mathbbm{Q}^{s,\eta},\mathbbm{F}^{s,\eta})$-square integrable martingales vanishing at time $s$, hence on the interval  $[0,s]$ since it is $\mathbbm{Q}^{s,\eta}$ a.s. deterministic on $[0,s[$,   see Proposition \ref{CoroTrivial}.
	The elements of that space will
	be identified 
	up to indistinguishability with respect to $\mathbbm{Q}^{s,\eta}$.
	
	We define $\mathcal{L}^2_{uni}$ as the linear space of $\mathbbm{F}^o$-progressively measurable processes belonging to $\mathcal{L}^2(dt\otimes d\mathbbm{Q}^{s,\eta})$ for all $(s,\eta)\in[0,T]\times\Omega$.
	Let  $\mathcal{N}$ be the linear subspace of $\mathcal{L}^2_{uni}$ constituted of elements which are equal to $0$ 
	$dt\otimes d\mathbbm{Q}^{s,\eta}$ a.e. for all $(s,\eta)\in[0,T]\times\Omega$. We denote  $L^2_{uni}:=\mathcal{L}^2_{uni}\backslash \mathcal{N}$.
\end{notation}

\begin{definition}\label{C7zeropotential} 
	A property will be said to hold \textbf{quasi surely},  abbreviated by q.s. if it holds everywhere but  in some $C\in \mathcal{P}ro^o$ such that $\mathbbm{E}^{s,\eta}\left[\int_s^{T} \mathds{1}_C(t,\omega)dt\right]=0$ for all $(s,\eta)$.
\end{definition}

\begin{notation}\label{mT}
From now on, $m^T:=(m^{T,1},\cdots,m^{T,d})$ will denote what we call 
\textbf{prediction process}
$(t,\omega)\mapsto m_t[\omega](T) = \mathbbm{E}^{t,\omega}[X_T]-\beta_T$,
in agreement with the 
last statement of Remark \ref{ResumeQ} 
 and Definition \ref{D42}.
 For all $(s,\eta)$ we 
introduce the $\mathbbm{R}^d$-valued $\mathbbm{Q}^{s,\eta}$-martingale 
$$m^{T,s,\eta}:=(m^{T,1,s,\eta},\cdots,m^{T,d,s,\eta}):=\int_{s}^{\cdot}k(T,r)dB^{s,\eta}_r,$$
 indexed by $[s,T]$ (extended by convention on $[0,s]$ with the value $0$), where we recall that $B^{s,\eta}$ is the Brownian motion introduced in Remark \ref{ResumeQ}.

\end{notation}


\begin{proposition}\label{mTprop}
	For all $i\leq d$, we have the following.
\begin{enumerate}

\item $m^{T,i}$ and $(m^{T,i})^2$ belong to $\mathcal{D}(A)$ with $A(m^{T,i})\equiv 0$ and 
$$\Gamma(m^{T,i})=\underset{j\leq d}{\sum}k_{i,j}^2(T,\cdot)=(kk^{\intercal})_{i,i}(T,\cdot), $$ which is  bounded;
\item 
for every $\tilde \Phi \in \tilde{\mathcal{D}}(\tilde A)$ we have 
 $\Gamma(\Phi,m^{T,i})_t(\omega)=\underset{j\leq d}{\sum}k_{i,j}(T,t)\nabla_{k_j(\cdot,t)}\tilde{\Phi}_{t}(m_t[\omega]),$ for all $t,\omega$;
\item for every $(s,\eta)$,  on $[s,T]$,
 $m^{T,i,s,\eta}$ is  $\mathbbm{Q}^{s,\eta}$-indistinguishable from\\ $m^{T,i}-m^{T,i}_s(\eta)$; it belongs to $\mathcal{H}^2_0(\mathbbm{Q}^{s,\eta})$ and  $\frac{d\langle m^{T,i,s,\eta}\rangle_t}{dt}=\underset{j\leq d}{\sum}k^2_{i,j}(T,t)$ is bounded $dt\otimes d\mathbbm{Q}^{s,\eta}$ a.e.;
\item $m^{T,i}\in\mathcal{L}^2_{uni}$.
\end{enumerate}
\end{proposition}
\begin{remark}\label{RX0}
  Let $1 \le i \le d$.
For all $(t,\omega)$, by definition of canonical process, we get
$m^{T,i}_t(\omega) = X^i_T(m_t[\omega])$.
So $(t,\omega) \mapsto X^i_T(\omega)$ is the  $\tilde \Phi$ corresponding to $\Phi = m^{T,i}$.
    \end{remark}

\begin{proof}
We fix $i\leq d$.

 By Example 3.5 1. in \cite{ViensZh} 
both
 $(t,\omega)  \mapsto X^i_T(\omega), (X^i_T)^2(\omega)$ belong to $\mathcal{C}_+^{1,2}(\bar{\Lambda})$; also
 $D X^i_T\equiv 0$, $\nabla_{b(\cdot,t)} X^i_T(\omega)= b_i(T,t)$, $\nabla_{k_j(\cdot,t)} X^i_T(\omega)= k_{i,j}(T,t)$ for all $j\leq d; (t,\omega)$ and that $\nabla^2_{k_j(\cdot,t)}(X^i_T)_t(\omega)=0$ for all $j\leq d, (t,\omega)$.
Therefore, according to Definition \ref{DefDA},  $m^{T,i}$ belongs to $\mathcal{D}(A)$ and $A(m^{T,i})\equiv 0$. 
Since $(X^i_T)^2\in\mathcal{C}_+^{1,2}(\bar{\Lambda})$, then  by Proposition \ref{ExprGamma},
$\Gamma(m^{T,i})$
 is equal to $\underset{j\leq d}{\sum}k_{i,j}^2(T,\cdot)$ which is bounded by Hypothesis \ref{Hypbk}. This shows item 1.

The second statement also holds by Proposition \ref{ExprGamma}, and the fact 
that \\
$\nabla_{k_j(\cdot,t)} X^i_T(\omega)= k_{i,j}(T,t)$ for all $i,j\leq d$ and $(t,\omega)$.

Concerning the third statement, we have $X^i_T=m^{T,i}_s(\eta)+\beta^i_T+\int_{s}^{T}k_i(T,r)dB^{s,\eta}_r$ $\mathbbm{Q}^{s,\eta}$ a.s. (see Remark \ref{ResumeQ}) so that for all $t\geq s$,
taking into account Notation \ref{mT}
\begin{equation}
\begin{array}{rcl}
m^{T,i}_t(\omega)&=&\mathbbm{E}^{t,\omega}[X^i_T]-\beta^i_T\\
&=&\mathbbm{E}^{s,\eta}[X^i_T|\mathcal{F}^o_t](\omega)-\beta^i_T \quad \mathbbm{Q}^{s,\eta} \text{a.s.}\\
&=&m^{T,i}_s(\eta)+\beta^i_T-\beta^i_T+\mathbbm{E}^{s,\eta}[\int_{s}^{T}k_i(T,r)dB^{s,\eta}_r|\mathcal{F}^o_t](\omega) \quad \mathbbm{Q}^{s,\eta} \text{a.s.}\\
&=&m^{T,i}_s(\eta)+\int_{s}^{t}k_i(T,r)dB^{s,\eta}_r\quad \mathbbm{Q}^{s,\eta} \text{a.s.},
\end{array}
\end{equation}
where the second equality holds because $(\mathbbm{Q}^{s,\eta})$
is a path-dependent canonical class, taking into account 
  Remark \ref{C7Borel}. So, according to Notation 
 \ref{mT} $m^{T,i,s,\eta}$ is a 
$\mathbbm{Q}^{s,\eta}$-modification of $m^{T,i}-m^{T,i}_s(\eta)$.
 Since $m^{T,i,s,\eta}$ is $\mathbbm{Q}^{s,\eta}$-a.s. continuous and $t\mapsto m^{T,i}_t(\omega)-m^{T,i}_s(\eta)$ is right-continuous for all $\omega$ (see Proposition \ref{mcad}), then those processes are indistinguishable. 
The quadratic variation 
 of $m^{T,i,s,\eta}$ is $\underset{j\leq d}{\sum}\int_s^{\cdot}k_{i,j}^2(T,r)dr$ so $\frac{d\langle m^{T,i,s,\eta}\rangle_t}{dt}=\underset{j\leq d}{\sum}k_{i,j}^2(T,t)$ $dt\otimes d\mathbbm{Q}^{s,\eta}$ a.e. is indeed bounded a.e. since $k$ is bounded. $m^{T,i,s,\eta}$ is 
a square integrable martingale,  because its quadratic variation is bounded.

We now discuss the last statement. For all $(s,\eta)$, $m^{T,i,s,\eta}$ is on $[s,T]$ $\mathbbm{Q}^{s,\eta}$-indistinguishable from $m^{T,i}-m^{T,i}_s(\eta)$ therefore $m^{T,i}\in\mathcal{L}^2_{uni}$ if and only if for every $(s,\eta)$, $m^{T,i,s,\eta}\in\mathcal{L}^2(dt\otimes d\mathbbm{Q}^{s,\eta})$. This indeed holds,
since for every $(s,\eta)$, by statement 3., $m^{T,i,s,\eta}$ is a 
square integrable martingale, hence $\underset{r\in[s,T]}{\text{sup }}|m^{T,i,s,\eta}_r|\in\mathcal{L}^2(\mathbbm{Q}^{s,\eta})$ by Doob inequality.
\end{proof}

\subsection{The semilinear path-dependent PDE and associated BSDE}\label{S44}

We now introduce the path-dependent PDE that interests us. 
We consider some $\xi,f$ verifying the following hypothesis.
\begin{hypothesis}\label{C7HypBSDE}\
	\begin{enumerate}
		\item $\xi$ is a r.v. with polynomial growth;
		\item $f:([0,T]\times\Omega)\times\mathbbm{R}\times\mathbbm{R}\longmapsto \mathbbm{R}$ is measurable with respect to $\mathcal{P}ro^o\otimes\mathcal{B}(\mathbbm{R})\otimes\mathcal{B}(\mathbbm{R})$ and such that 
		\begin{enumerate}
			\item $f(\cdot,\cdot,0,0)$ has polynomial growth;
			\item there exists $K>0$ such that for all $(t,\omega,y,y',z,z')\in [0,T]\times\Omega\times\mathbbm{R}\times\mathbbm{R}\times\mathbbm{R}\times\mathbbm{R}$
			\begin{equation}
			|f(t,\omega,y',z')-f(t,\omega,y,z)|\leq K(|y'-y|+|z'-z|).
			\end{equation}
		\end{enumerate}
	\end{enumerate}
\end{hypothesis}
We recall that the notion of polynomial growth has been introduced
in  Definition \ref{DefViens}.

\begin{remark} \label{R424}
A direct consequence of Proposition \ref{supX} and the fact that $\mathbbm{Q}^{s,\eta}$ is a translation of $\mathbbm{P}^{s,\eta}$, is that 
since 
$\xi,f$
 verify Hypothesis \ref{C7HypBSDE}, then $\xi$ belongs to $\mathcal{L}^2(\mathbbm{Q}^{s,\eta})$ for all $(s,\eta)$ and $f(\cdot,\cdot,0,0)\in\mathcal{L}^2_{uni}$.
We remark that $\mathbbm{Q}^{s,\eta}$ is just a translation of
 $\mathbbm{P}^{s,\eta}$, under which $\|\omega\|_{\infty}$ admits finite moments of every order, see Proposition \ref{supX}.    

\end{remark}

We now consider the following abstract path-dependent non linear equation.

\begin{equation}\label{C7AbstractEq}
\left\{
\begin{array}{l}
A\Phi+f(\cdot,\cdot,\Phi,\Gamma(m^T,\Phi))=0\quad \text{ on }[0,T]\times\Omega\\
\Phi_T=\xi\quad\text{ on }\Omega.
\end{array}\right.
\end{equation}

\begin{remark}\label{RemPDE}
In previous equation \eqref{C7AbstractEq}, if
$\tilde \Phi \in {\mathcal D}(\tilde A)$
 then the equation can also be written
\begin{equation}
\left\{
\begin{array}{l}
\left(D\tilde{\Phi}_{t}+\nabla_{b(\cdot,t)}\tilde{\Phi}_{t} +\frac{1}{2}\underset{i\leq d}{\sum}\nabla_{k_i(\cdot,t)}^2\tilde{\Phi}_{t} + f\left(t,\cdot,\tilde{\Phi}_{t},k(T,t)\nabla_{k(\cdot,t)}\tilde{\Phi}_{t}\right)\right)(m_t[\omega])=0,  \\ 
\quad \quad \quad \text{ on }[0,T]\times\Omega\\
\tilde{\Phi}_{T}(\omega)=\xi(\omega)\quad\text{ on }\Omega,
\end{array}\right.
\end{equation}
see Definition \ref{DefDA} and Proposition \ref{mTprop},
observing that $m_T(\omega) = \omega$, hence \\
$\tilde \Phi_T = \Phi_T$.
\end{remark}
\begin{notation} \label{N425}
Equation \eqref{C7AbstractEq} will be denoted $PDPDE(f,\xi)$.
\end{notation}

\begin{definition}\label{DefSol}
A process $Y$ will be called a \textbf{classical solution of }$PDPDE(f,\xi)$ if it belongs to $\mathcal{D}(A)$
and if $Y$ verifies \eqref{C7AbstractEq}.
A process $Y\in\mathcal{L}^2_{uni}$ will be called  \textbf{decoupled mild solution of} $PDPDE(f,\xi)$ if there exist auxiliary processes $Z^1,\cdots,Z^d\in\mathcal{L}^2_{uni}$ such that for all $(s,\eta)\in[0,T]\times\Omega$:
		\begin{equation}\label{C7AbMildEq}
		\left\{
		\begin{array}{rl}
		Y_s(\eta)&=P_s[\xi](\eta)+\int_s^TP_s\left[f\left(r,\cdot,Y_r,Z_r\right)\right](\eta)dr\\
		(Ym^T)_s(\eta) &=P_s[\xi m^T_T](\eta) -\int_s^TP_s\left[\left(Z_r-m^T_rf\left(r,\cdot,Y_r,Z_r\right)\right)\right](\eta)dr,
		\end{array}\right.
		\end{equation}
where $Z:=(Z^1,\cdots,Z^d)$.

The couple $(Y,Z)$ will be said to \textbf{solve the identification problem} $IP(f,\xi)$. 
\end{definition}
Decoupled mild solutions were introduced in path-dependent framework in \cite{paperPathDep} and in the  framework of classical parabolic PDEs in \cite{paper3}.

To $PDPDE(f,\xi)$ we associate the following family of BSDEs  indexed by $(s,\eta)$ and defined on the time interval $[s,T]$:
\begin{equation} \label{C7BSDE}
Y^{s,\eta}_{\cdot}=\xi+\int_{\cdot}^{T}f\left(r,\cdot,Y^{s,\eta}_r,\frac{d\langle M^{s,\eta},m^{T,s,\eta}\rangle_r}{dr}\right)dr-(M^{s,\eta}_T-M^{s,\eta}_{\cdot}),
\end{equation} 
in the stochastic basis $\left(\Omega,\mathcal{F}^{s,\eta},\mathbbm{F}^{s,\eta},\mathbbm{Q}^{s,\eta}\right)$, where $m^{T,s,\eta}$, which was introduced in Notation \ref{mT}, is the driving martingale. \\  
\begin{remark} \label{R427}
Taking into account Hypothesis \ref{C7HypBSDE}, Remark \ref{R424}
 and the fact that  $\frac{d\langle m^{T,s,\eta}\rangle_t}{dt}$ is bounded $dt\otimes d\mathbbm{Q}^{s,\eta}$ a.e. (see Proposition \ref{mTprop}), 
 Theorem 3.3 and Remark 3.4 in \cite{paper3} applied with $\hat{M}:=m^{T,s,\eta};\quad V_t\equiv t$ imply that for every $(s,\eta)$, there exists a unique couple $(Y^{s,\eta},M^{s,\eta})\in\mathcal{L}^2(dt\otimes d\mathbbm{Q}^{s,\eta})\times \mathcal{H}^2_0(\mathbbm{Q}^{s,\eta})$ verifying \eqref{C7BSDE} on $[s,T]$.
\end{remark}

We state now the main results of this paper.
\begin{theorem}\label{MainTheorem}
Assume the validity of  Hypotheses \ref{HypGaussProba} for $\mathbbm{P}$, 
\ref{Hypbk} for $b,k$ and \ref{C7HypBSDE} for $\xi,f$.
\begin{enumerate}
\item $PDPDE(f,\xi)$ has a unique decoupled mild solution;
\item $IP(f,\xi)$ admits a unique solution
 $(Y,Z)\in\mathcal{L}^2_{uni}\times (L^2_{uni})^d$.
By uniqueness we mean more precisely the following: 
if $(Y,Z)$ and $(\bar Y, \bar Z)$ are two solutions then
$Y $ and $\bar Y$ are identical and $Z = \bar Z$ q.s.
\item For every $(s,\eta)$, let 
  $(Y^{s,\eta},M^{s,\eta})$ be the solution  of \eqref{C7BSDE}.
  Then, for every $(s,\eta)$, we have that, $Y^{s,\eta}$ is on $[s,T]$ a $\mathbbm{Q}^{s,\eta}$ modification of $Y$, and $Z_t=\frac{d\langle M^{s,\eta},m^{T,s,\eta}\rangle_t }{dt}$ $dt\otimes d\mathbbm{Q}^{s,\eta}$ a.e. In particular, $Y_s(\eta)= Y^{s,\eta}_s$.
\end{enumerate}
\end{theorem}

\begin{proof}
We make use of Theorem 3.19 in \cite{paperPathDep} applied with the operator $(\mathcal{D}(A),A)$ introduced in Definition \ref{DefDA}  and with $\Psi:=m^T$.

$(\mathbbm{Q}^{s,\eta})_{(s,\eta)\in[0,T]\times\Omega}$ is a progressive path-dependent canonical class verifying Hypothesis \ref{HypClass} (see Propositions \ref{ClassGauss}, \ref{Prog} and Corollary \ref{CoroHyp}) and $(\mathcal{D}(A),A)$ is weak generator of $(P_s)_{s\in[0,T]}$ (see Proposition \ref{PropWeakGen}) as required 
in Section 3.6 of \cite{paperPathDep}.
$\Psi:=m^T$ verifies 
Hypothesis 3.16
 in \cite{paperPathDep} thanks to Proposition \ref{mTprop}, and $\xi,f$ verify
Hypothesis 3.6 in \cite{paperPathDep} thanks to Hypothesis \ref{C7HypBSDE} and Remark \ref{R424}. So 
Theorem 3.19 in \cite{paperPathDep} applies.
\end{proof}

The link between decoupled mild solutions and classical solutions is the following.

\begin{proposition}\label{C7classical}
	\begin{enumerate}\
		\item Let $\Phi$ be a classical solution of $PDPDE(f,\xi)$, 
see Definition \ref{DefSol}. Then $(\Phi,\Gamma(m^T,\Phi))$ is a solution of the identification problem $IP(f,\xi)$ (see Definition \ref{DefSol}) and in particular, $\Phi$ is a decoupled mild solution of $PDPDE(f,\xi)$;
		\item there is at most one classical solution of $PDPDE(f,\xi)$;
		\item assume that the unique decoupled mild solution $Y$ of $PDPDE(f,\xi)$ verifies $Y\in\mathcal{D}(A)$;
		then $Y$ is a classical solution q.s., in the sense that $Y_T=\xi$ (for all $\omega$) and that $A(Y)=-f(\cdot,\cdot,Y,\Gamma(m^T,Y))$ q.s., see
 Definition \ref{C7zeropotential}. 
	\end{enumerate}
\end{proposition}
\begin{proof}\
\begin{enumerate}
\item	
 Let $\Phi:(t,\omega)\longmapsto\tilde{\Phi}_{t}(m_t[\omega])$ be a classical solution. First, 
since $\Phi$   belongs to $\mathcal{D}(A)$ then $\Phi$ and 
$$\Gamma(m^{T,i},\Phi):(t,\omega)\longmapsto \underset{j\leq d}{\sum}k_{i,j}(T,t)\nabla_{k_j(\cdot,t)}\tilde{\Phi}_{t}(m_t[\omega]),i\leq d,$$   have polynomial growth, see Definition 
\ref{DefDA}. Hence thanks to Proposition \ref{supX} and since $\mathbbm{Q}^{s,\eta}$ is a translation of $\mathbbm{P}^{s,\eta}$ , those processes belong to $\mathcal{L}^2_{uni}$.

On the other hand, let  $(s,\eta)\in[0,T]\times\Omega$.
By Proposition \ref{PropWeakGen},
\begin{eqnarray}\label{EMMart}
M[\Phi]^{s,\eta}  &=& \Phi-\Phi_s(\eta)-\int_s^{\cdot}A\Phi_rdr, \\
M[\Phi m^T]^{s,\eta}&=& \Phi m^T-\Phi_s(\eta)m^T_s(\eta)-\int_s^{\cdot}A(\Phi m^T)_rdr,
\end{eqnarray} 
 are $\mathbbm{Q}^{s,\eta}$-martingales on $[s,T]$ 
vanishing at time $s$. 
By Definition \ref{DefSol} we have $A\Phi=-f(\cdot,\cdot,\Phi,\Gamma(m^T,\Phi))$ and by  Propositions \ref{ExprGamma}, \ref{mTprop}, taking into account
 Notation \ref{C7NotGamma},
we have 
$$A(\Phi m^T)=\Gamma(\Phi, m^T)+\Phi Am^T+m^TA\Phi=\Gamma(\Phi, m^T)-m^Tf(\cdot,\cdot,\Phi,\Gamma(m^T,\Phi)),$$
 so the martingales \eqref{EMMart},
 indexed by $[s,T]$, 
can be rewritten as
\small
\begin{equation}\label{C7PreviousMart}
\left\{\begin{array}{rcl}
M[\Phi]^{s,\eta} &=&  \Phi-\Phi_s(\eta)+\int_s^{\cdot}f(r,\cdot,\Phi_r,\Gamma(m^T,\Phi)_r)dr\\
M[\Phi m^T]^{s,\eta} &=& 
  \Phi m^T-\Phi_s(\eta)m^T_s(\eta)-\int_s^{\cdot}(\Gamma(m^T,\Phi)_r-m^T_rf(r,\cdot,\Phi,\Gamma(m^T,\Phi)_r)dr.
\end{array}\right.
\end{equation}
 \normalsize
 Finally, again by Definition \ref{DefSol} we have  $\Phi_T=\xi$, so, for any $(s,\eta)$,
 taking the expectations in \eqref{C7PreviousMart} at $s = T$, we get  
	\small
	\begin{equation}
	\left\{
	\begin{array}{l}
	\mathbbm{E}^{s,\eta}\left[\xi-\Phi_s(\eta)+\int_s^{T}f(r,\cdot,\Phi_r,\Gamma(m^T,\Phi)_r)dr\right]=0;\\
	\mathbbm{E}^{s,\eta}\left[\xi m^T_T-(\Phi m^T)_s(\eta)-\int_s^{T}(\Gamma(m^T,\Phi)_r-m^T_rf(r,\cdot,\Phi,\Gamma(m^T,\Phi)_r))dr\right]=0,
	\end{array}
	\right.
	\end{equation}
	\normalsize
	which by Fubini's theorem and Definition \ref{ProbaOp} 
 yields
	\small
	\begin{equation}
	\left\{
	\begin{array}{rl}
	\Phi_s(\eta)=&P_s[\xi](\eta)+\int_s^TP_s\left[f\left(r,\cdot,\Phi_r,\Gamma(m^T,\Phi)_r\right)\right](\eta)dr\\
	(\Phi m^T)_s(\eta) =&P_s[\xi m^T_T](\eta) -\int_s^TP_s\left[\Gamma(m^T,\Phi)_r-m^T_rf\left(r,\cdot,\Phi_r,\Gamma(m^T,\Phi)_r\right)\right](\eta)dr
	\end{array}\right.
	\end{equation}
	\normalsize
	and the first item is proved.
\item
	The second item follows from item 1. and by the uniqueness of a decoupled mild solution of $PDPDE(f,\xi)$, see Theorem \ref{MainTheorem} 1. 
	\item
	Concerning item 3. let $(Y,Z)$ be the unique  solution of $IP(f,\xi)$. We first note again that the first line of \eqref{C7AbMildEq} taken with $s=T$ yields $Y_T=\xi$. 
	
        Let us now fix some $(s,\eta)\in[0,T]\times\Omega$. The fact that $Y\in\mathcal{D}(A)$, Proposition \ref{PropWeakGen}  
        implies  that
 $M^{s,\eta}[Y]=Y-Y_s-\int_s^{\cdot}AY_rdr$ is on $[s,T]$ a continuous $\mathbbm{Q}^{s,\eta}$-martingale. Hence  $Y$ is under $\mathbbm{Q}^{s,\eta}$ a continuous semi-martingale.
Let us keep in mind the unique solution  $(Y^{s,\eta}, M^{s,\eta})  $ 	
of \eqref{C7BSDE}.
A consequence of item 3. of  
Theorem \ref{MainTheorem}
is that $Y$ admits on $[s,T]$, $Y^{s,\eta}$ as $\mathbbm{Q}^{s,\eta}$ cadlag version  
 which is a special semi-martingale verifying  
$ Y^{s,\eta}_t = Y_s^{s,\eta} - \int_s^tf(r,Y_r,Z_r)dr + M^{s,\eta}_t, t \in [s,T]$.
 Since $Y$ is $\mathbbm{Q}^{s,\eta}$ a.s. continuous, then $Y$ and $Y^{s,\eta}$ are actually $\mathbbm{Q}^{s,\eta}$-indistinguishable on $[s,T]$; so the uniqueness of the decomposition of 
the semi-martingale $Y$, yields that $(\int^{\cdot}_sAY_rdr,M^{s,\eta}[Y])$ and $(-\int^{\cdot}_sf(r,Y_r,Z_r)dr,M^{s,\eta})$ are
  $\mathbbm{Q}^{s,\eta}$-indistinguishable on $[s,T]$. Since this holds for all
 $(s,\eta)$, by Definition \ref{C7zeropotential} we have
 $AY=-f(\cdot,\cdot,Y,Z)$ q.s. so we are left to show that $Z=\Gamma(m^T,Y)$ q.s.

For this we fix again $(s,\eta)$. By item 3. of Theorem \ref{MainTheorem},
 $\langle M^{s,\eta},m^{T,s,\eta}\rangle = \int_s^{\cdot}Z_rdr$. 
 By Corollary \ref{CoroBracketGamma} and Propositions \ref{ExprGamma}, \ref{mTprop}, we have  
$$\langle M^{s,\eta}[Y],m^{T,s,\eta}\rangle = \int_s^{\cdot}\Gamma(m^T,\Phi)_rdr.$$ As we  have remarked above,
 $M^{s,\eta}=M^{s,\eta}[Y]$ so $\int_s^{\cdot}Z_rdr=\langle M^{s,\eta},m^{T,s,\eta}\rangle=\langle M^{s,\eta}[Y],m^{T,s,\eta}\rangle=\int_s^{\cdot}\Gamma(m^T,\Phi)_rdr$  under $\mathbbm{Q}^{s,\eta}$ on $[s,T]$. Since this holds for all $(s,\eta)$, we indeed have by Definition \ref{C7zeropotential} that $Z=\Gamma(m^T,\Phi)$ q.s., and the proof is complete.
\end{enumerate}

\end{proof}

\begin{appendix}
	
\section{Technical proofs of Section \ref{SGauss}}
	
\begin{prooff}. of Proposition \ref{ClassGauss}.
	
	The fact that item 1. of Definition \ref{DefCondSyst} holds comes from
	item 1. of Proposition \ref{CondExp}. We now show that item 2. of Definition \ref{DefCondSyst} holds. We fix $s\in[0,T]$. By item 2. (b) of Proposition \ref{CondExp}, $\begin{array}{rcl}
	\eta & \longmapsto & \mathbbm{P}^{s,\eta}\\
	\Omega_s & \longrightarrow & \mathcal{P}(\Omega)
	\end{array}$ is continuous hence Borel, and by Proposition 5.3 in \cite{paperMPv2}, $\mathcal{B}(\Omega_s)=\Omega_s\cap\mathcal{F}^o_s$. Since $\pi_s:\Omega\rightarrow\Omega_s$ is trivially $(\mathcal{F}^o_s, \Omega_s\cap\mathcal{F}^o_s)$-measurable, taking Definition \ref{extension} into account, it follows
that $\begin{array}{rcl}
	\eta & \longmapsto & \mathbbm{P}^{s,\eta}\\
	\Omega & \longrightarrow & \mathcal{P}(\Omega)
	\end{array}$ is $\mathcal{F}^o_s$-measurable; so item 2. of Definition \ref{DefCondSyst} holds.

	We are left to show that its item 3. also holds. We fix $0\leq s\leq t\leq T$. We recall that there exists a countable $\pi$-system $\Pi_T$ (resp. $\Pi_t$) generating $\mathcal{F}$
	(resp. $\mathcal{F}^o_t$), see
 Notation 5.3 in \cite{paperMPv2} for instance. We fix $F\in\Pi_T$ and $G\in\Pi_t$ and we get
	\begin{equation}
	\begin{array}{rl}
	&\mathbbm{E}^{s,\eta}[\mathbbm{P}^{t,\zeta}(F)\mathds{1}_G(\zeta)]\\
	=&\mathbbm{E}^{s,\eta}[\mathbbm{E}^{t,\zeta}[\mathds{1}_F(\omega)]\mathds{1}_G(\zeta)]\\
	=&\mathbbm{E}^{s,\eta}[\mathbbm{E}^{t,\zeta}[\mathds{1}_F(\omega)\mathds{1}_G(\zeta)]]\\
	=&\mathbbm{E}^{s,\eta}[\mathbbm{E}^{t,\zeta}[\mathds{1}_F(\omega)\mathds{1}_G(\omega)]]\\
	=&\mathbbm{E}[\mathbbm{E}[\mathds{1}_F\mathds{1}_G|\mathcal{F}^o_t]|\mathcal{F}^o_s](\eta),\quad \mathbbm{P}\text{ a.s.}\\
	=&\mathbbm{E}[\mathds{1}_F\mathds{1}_G|\mathcal{F}^o_s](\eta),\quad \mathbbm{P}\text{ a.s.}\\
	=&\mathbbm{E}^{s,\eta}[\mathds{1}_F\mathds{1}_G],\quad \mathbbm{P}\text{ a.s.},
	\end{array}
	\end{equation}
	where the third equality holds because $G\in\mathcal{F}^o_t$ and $\mathbbm{P}^{t,\zeta}(\omega^t=\zeta^t)=1$ therefore $\mathds{1}_G=\mathds{1}_G(\zeta)$ $\mathbbm{P}^{t,\zeta}$-a.s.; the fourth and sixth equalities  hold by \eqref{regcondexp}. Since $\Pi_t$ and $\Pi_T$ are countable,
there is  a set $\mathcal{N}^c$ of $\mathbb P$-full measure,
such that for
	all $\eta$ in a set $\mathcal{N}^c$ 
	\begin{equation}\label{Eq1}
	\mathbbm{E}^{s,\eta}[\mathbbm{P}^{t,\omega}(F)\mathds{1}_G(\omega)]=\mathbbm{E}^{s,\eta}[\mathds{1}_F\mathds{1}_G],
	\end{equation}
	for all $F\in\Pi_T, G\in\Pi_t$. By a monotone class argument, 
for all $\eta\in\mathcal{N}^c$, \eqref{Eq1} holds for all $F\in \mathcal{F},G\in\mathcal{F}^o_t$. Therefore for every $\eta\in\mathcal{N}^c$, $(\mathbbm{P}^{t,\omega})_{\omega\in\Omega}$ verifies \eqref{DE13}.

	We will now show that $(\mathbbm{P}^{t,\omega})_{\omega\in\Omega}$ verifies \eqref{DE13} for all $\eta\in\Omega$ and not just for $\eta\in\mathcal{N}^c$. 
	Since $\mathcal{N}^c$ is of full measure, then its closure is a closed set of full measure hence is equal to $\Omega$ by Hypothesis \ref{HypGaussProba} item 2., so $\mathcal{N}^c$ is dense in $\Omega$. We fix $\eta\in\Omega$, a sequence $(\eta_n)_n$ of elements of $\mathcal{N}^c$ converging to $\eta$, some $\Phi\in\mathcal{C}_b(\Omega,\mathbbm{R})$ and some $\mathcal{F}^o_t$-measurable $\Psi^t\in\mathcal{C}_b(\Omega,\mathbbm{R})$. For 
	every $n$, since  $\eta_n\in\mathcal{N}^c$,
 then $(\mathbbm{P}^{t,\omega})_{\omega\in\Omega}$ verifies \eqref{DE13} with $(s,\eta)$ replaced by $(s,\eta_n)$, so 
	\begin{equation}\label{Eq2}
	\mathbbm{E}^{s,\eta_n}[\Phi\Psi^t]= \mathbbm{E}^{s,\eta_n}[\mathbbm{E}^{t,\omega}[\Phi]\Psi^t(\omega)].
	\end{equation}
	$\Phi,\Psi^t$ are bounded continuous. By Proposition \ref{CondExp} item 2.(b), $\omega\mapsto\mathbbm{P}^{t,\omega}$ is continuous; 
by definition of the topology on $\mathcal{P}(\Omega)$, $\omega\mapsto\mathbbm{E}^{t,\omega}[\Phi]$ is (bounded) continuous,  therefore
 $\mathbbm{E}^{t,\cdot}[\Phi]\Psi^t$ is bounded continuous; moreover
 $\Phi\Psi^t$ is also bounded and continuous.
	Since $\eta\mapsto\mathbbm{P}^{s,\eta}$ is continuous, we can pass to the limit in $n$ in \eqref{Eq2} and we get 
	\begin{equation}\label{Eq3}
	\mathbbm{E}^{s,\eta}[\Phi\Psi^t]= \mathbbm{E}^{s,\eta}[\mathbbm{E}^{t,\omega}[\Phi]\Psi^t(\omega)].
	\end{equation}
	\eqref{Eq3} holds for all bounded continuous $\Phi$ and bounded continuous $\mathcal{F}^o_t$-measurable $\Psi^t$ so by the functional monotone class theorem (see Theorem 21 in \cite{dellmeyer75} Chapter I), it holds for all bounded measurable $\Phi$ and bounded $\mathcal{F}^o_t$-measurable $\Psi^t$. Since this is true for any $\eta$ we have shown in particular that for all $\eta$, $(\mathbbm{P}^{t,\omega})_{\omega\in\Omega}$ verifies \eqref{DE13}, i.e. item 3. of Definition \ref{DefCondSyst}.
\end{prooff}


\begin{remark} \label{R314}
For given $s \in ]0,T]$,	
	$\Omega_s^*$ can be characterized as the set of bounded positive measures on $[0,s]$ not charging $0$.

	The action of $\pi_s^*:\Omega_s^*\longrightarrow\Omega^*$ is given by    $\pi_s^*(\mu):\omega\longmapsto\mu(\omega|_{[0,s]})$ for all $\mu\in\Omega_s^*$.
\end{remark}
\begin{prooff}. of Proposition \ref{mcad}.
	
	We fix $\eta\in\Omega$ and $s\in[0,T[$ and we show that $m_{\cdot}[\eta]$ is right-continuous in $s$.

	We recall that 
	$\pi_sK\pi_s^{*}\Omega_s^{*}$ is dense in $\Omega_s$ and that on $K\pi_s^{*}\Omega_s^{*}$, $m_s\circ\pi_s$ coincides with the identity, see Proposition \ref{CondExp} item 2.(f).
	Therefore since for all $\delta\geq 0$, we clearly have 
$\pi_s^{*}\Omega_s^{*}\subset \pi_{s+\delta}^{*}\Omega_{s+\delta}^{*}$, so $K\pi_s^{*}\Omega_s^{*}\subset K\pi_{s+\delta}^{*}\Omega_{s+\delta}^{*}$  then for all $\epsilon\geq 0$ there exists $\eta_{\epsilon}\in K\pi_s^{*}\Omega_s^{*}$ such that 
		\begin{eqnarray}\label{Eq4}
	&& \|\pi_s(\eta)-\pi_s(\eta_{\epsilon})\|_{\infty}\leq \frac{\epsilon}{4M_{op}}; \nonumber\\
	&&  m_s\pi_s(\eta_{\epsilon})=\eta_{\epsilon}; \\
	&&  {\rm for \ all} \ \delta \geq 0, m_{s+\delta}\pi_{s+\delta}(\eta_{\epsilon})=\eta_{\epsilon}.
	\nonumber
	\end{eqnarray}
	
	For all $\delta\geq 0$, by Definition \ref{extension} we write
	\begin{equation}
	\begin{array}{rl}
	&m_{s+\delta}(\eta)-m_s(\eta)\\
	=& m_{s+\delta}\pi_{s+\delta}(\eta)-m_s\pi_s(\eta)\\
	=&(m_{s+\delta}\pi_{s+\delta}(\eta)-m_{s+\delta}\pi_{s}(\eta))+(m_{s+\delta}\pi_{s}(\eta)-m_{s+\delta}\pi_{s}(\eta_{\epsilon}))\\
	&+ (m_{s+\delta}\pi_{s}(\eta_{\epsilon})-m_{s+\delta}\pi_{s+\delta}(\eta_{\epsilon}))+(m_{s+\delta}\pi_{s+\delta}(\eta_{\epsilon})-m_s\pi_s(\eta_{\epsilon}))\\
	&+(m_s\pi_s(\eta_{\epsilon})-m_s\pi_s(\eta)),
	\end{array}
	\end{equation}
	where the fourth term of the sum is equal to zero since by \eqref{Eq4}, $m_s\pi_s(\eta_{\epsilon})=\eta_{\epsilon}=m_{s+\delta}\pi_{s+\delta}(\eta_{\epsilon})$. So we obtain
	\begin{equation}\label{Eq5}
	\begin{array}{rl}
	&\|m_{s+\delta}(\eta)-m_s(\eta)\|_{\infty}\\
	\leq &\|m_{s+\delta}\pi_{s+\delta}(\eta)-m_{s+\delta}\pi_{s}(\eta)\|_{\infty}+\|m_{s+\delta}\pi_{s}(\eta)-m_{s+\delta}\pi_{s}(\eta_{\epsilon})\|_{\infty}\\
	&+\|m_{s+\delta}\pi_{s}(\eta_{\epsilon})-m_{s+\delta}\pi_{s+\delta}(\eta_{\epsilon})\|_{\infty}+\|m_s\pi_s(\eta_{\epsilon})-m_s\pi_s(\eta)\|_{\infty}\\
	\leq &M_{op}(\|\pi_{s+\delta}(\eta)-\pi_{s}(\eta)\|_{\infty}+2\|\pi_{s}(\eta)-\pi_{s}(\eta_{\epsilon})\|_{\infty}+\|\pi_{s}(\eta_{\epsilon})-\pi_{s+\delta}(\eta_{\epsilon})\|_{\infty})\\
	\leq &\frac{\epsilon}{2}+M_{op}(\|\pi_{s+\delta}(\eta)-\pi_{s}(\eta)\|_{\infty}+\|\pi_{s}(\eta_{\epsilon})-\pi_{s+\delta}(\eta_{\epsilon})\|_{\infty},
	\end{array}
	\end{equation}
	where the second inequality holds by Proposition \ref{CondExp} 
	item 2. (e), and the third inequality by the first line of \eqref{Eq4}.
	Since clearly, for all $\omega\in\Omega$, $\pi_{s+\delta}(\omega)$ tends 
	uniformly to $\pi_s(\omega)$, then there exists $\delta$ small enough such that $\|\pi_{s+\delta}(\eta)-\pi_{s}(\eta)\|_{\infty}\leq \frac{\epsilon}{4M_{op}}$ and $\|\pi_{s+\delta}(\eta)-\pi_{s}(\eta)\|_{\infty}+\|\pi_{s}(\eta_{\epsilon})-\pi_{s+\delta}(\eta_{\epsilon})\|_{\infty}\leq \frac{\epsilon}{4M_{op}}$ which combined with \eqref{Eq5} gives
	\begin{equation}
	\|m_{s+\delta}(\eta)-m_s(\eta)\|_{\infty}\leq \epsilon;
	\end{equation}
 the right-continuity of $m_{\cdot}(\eta)$ at time $s$ is now proved. 
\end{prooff}

\begin{prooff}. of Proposition \ref{FullSupp}.
	
	We fix some $\eta_0\in\Omega$. It is obvious that $supp(\mathbbm{P}^{s,\eta_0})\subset \eta_0^s +\Omega_s^{\perp}$. We assume that there exists an open
	set of $\eta_0^s + \Omega_s^{\perp}$ which has $\mathbbm{P}^{s,\eta_0}$ zero measure.
	We will find a contradiction, and this will imply that
	$supp(\mathbbm{P}^{s,\eta_0})= \eta_0^s +\Omega_s^{\perp}$.
	$\eta_0^s +\Omega_s^{\perp}$ is a closed subset of $\Omega$.
	Suppose that there is an open subset  (with respect to the induced topology)
	of $\eta_0^s +\Omega_s^{\perp}$ having
	zero $(\mathbbm{P}^{s,\eta_0})$-measure.
	It necessarily   contains a set of type
	$\eta_0^s + B(\zeta, \delta)\cap \Omega_s^{\perp}$ where $B(\zeta, \delta)$ is the open ball (in $\Omega$) of center $\zeta\in\Omega$ and radius $\delta>0$.
	Then we have 
	\begin{equation}\label{EqOpen1}
	\begin{array}{rcl}
	&&\mathbbm{P}^{s,\eta_0}(\eta_0^s + B(\zeta, \delta)\cap \Omega_s^{\perp})=0\\
	&\Longleftrightarrow& \mathbbm{P}^{s,\eta_0}(\omega^s = \eta_0^s \text{ and } \pi_s^{\perp}(\omega)\in B(\zeta, \delta))=0\\
	&\Longleftrightarrow&\mathbbm{P}^{s,\eta_0}( \pi_s^{\perp}(\omega)\in B(\zeta, \delta))=0\\
	&\Longleftrightarrow&\mathbbm{P}^{s,\eta_0}( (\pi_s^{\perp})^{-1}\left(B(\zeta, \delta)\right))=0,
	\end{array}
	\end{equation}
	where we remark that,
        being $\pi_s^{\perp}$ continuous, $(\pi_s^{\perp})^{-1}\left(B(\zeta, \delta)\right)$ is an open set of $\Omega$.	
	By continuity of $m_s$, $m_s^{-1}(B(m_s[\eta_0],\frac{\delta}{2}))$ is also an open set of $\Omega$.
	
	Let $\eta\in m_s^{-1}(B(m_s[\eta_0],\frac{\delta}{2}))$. Then
	\begin{equation}\label{EqOpen2}
	\begin{array}{rcl}
	&&\mathbbm{P}^{s,\eta}( (\pi_s^{\perp})^{-1}\left(B(\zeta, \frac{\delta}{2})\right))\\
	&=& \mathbbm{P}^{s,\eta}( \omega - \omega^s\in B(\zeta, \frac{\delta}{2}))\\
	&=& \mathbbm{P}^{s,\eta}( \omega - \eta^s\in B(\zeta, \frac{\delta}{2}))\\
	&=& \mathbbm{P}^{s,\eta_0}( \omega - \eta^s+m_s[\eta]-m_s[\eta_0]\in B(\zeta, \frac{\delta}{2}))\\
	&=& \mathbbm{P}^{s,\eta_0}( \omega - \eta^s+(\eta^s + \pi_s^{\perp}(m_s[\eta]))-(\eta_0^s + \pi_s^{\perp}(m_s[\eta_0]))\in B(\zeta, \frac{\delta}{2}))\\
	&=& \mathbbm{P}^{s,\eta_0}( \omega - \eta_0^s+\pi_s^{\perp}(m_s[\eta]-m_s[\eta_0]))\in B(\zeta, \frac{\delta}{2}))\\
	&=&\mathbbm{P}^{s,\eta_0}( \omega - \omega^s+\pi_s^{\perp}(m_s[\eta]-m_s[\eta_0]))\in B(\zeta, \frac{\delta}{2}))\\
	&=&\mathbbm{P}^{s,\eta_0}( \pi_s^{\perp}(\omega) \in  B(\zeta, \frac{\delta}{2})- (m_s[\eta]-m_s[\eta_0]))\\
	&\leq &\mathbbm{P}^{s,\eta_0}( \pi_s^{\perp}(\omega) \in  B(\zeta, \delta))\\
	&=& 0,
	\end{array}
	\end{equation}
	where the third equality holds by Remark \ref{RPs}. The fourth is due to the fact that any $\omega$ can be decomposed in $\omega=\pi_s(\omega)+\pi^{\perp}_s(\omega)$ (see Notation \ref{Npis}) and that for all $\omega$, $m_s[\omega]$ and $\omega$ coincide on $[0,s]$ hence $\pi_s[m_s[\omega]]=\pi_s[\omega]=\omega^s$.
	The inequality holds because $B(\zeta, \frac{\delta}{2})- (m_s[\eta]-m_s[\eta_0]))\subset B(\zeta, \delta)$ since $\|m_s[\eta] - m_s[\eta_0]\|<\frac{\delta}{2}$, and the last equality  by \eqref{EqOpen1}.
	
	We can now consider the  set $m_s^{-1}(B(m_s[\eta_0],\frac{\delta}{2}))\cap (\pi_s^{\perp})^{-1}\left(B(\zeta, \delta)\right)$ which is open as intersection of open sets, and we compute 
	\begin{equation}
	\begin{array}{rcl}
	&&\mathbbm{P}\left(m_s^{-1}(B(m_s[\eta_0],\frac{\delta}{2}))\cap (\pi_s^{\perp})^{-1}\left(B(\zeta, \delta)\right)\right)\\
	&=&\mathbbm{E}\left[\mathbbm{E}\left[\mathds{1}_{m_s^{-1}(B(m_s[\eta_0],\frac{\delta}{2}))}\mathds{1}_{(\pi_s^{\perp})^{-1}\left(B(\zeta, \delta)\right)}|\mathcal{F}_s^o\right]\right]\\
	&=& \mathbbm{E}\left[\mathbbm{E}\left[\mathds{1}_{(\pi_s^{\perp})^{-1}\left(B(\zeta, \delta)\right)}\middle|\mathcal{F}_s^o\right]\mathds{1}_{m_s^{-1}(B(m_s[\eta_0],\frac{\delta}{2}))}\right]\\
	&=& \mathbbm{E}\left[\mathbbm{P}^{s,\cdot}\left((\pi_s^{\perp})^{-1}\left(B(\zeta, \delta)\right)\right)\mathds{1}_{m_s^{-1}(B(m_s[\eta_0],\frac{\delta}{2}))}\right]\\
	&=& 0,
	\end{array}
	\end{equation}
	where the second equality holds (taking into account Definition
        \ref{extension}), because $m_s$ is $\mathcal{F}_0^s$-measurable,
        and the last equality holds by \eqref{EqOpen2}.
	
	So the non-empty open set $m_s^{-1}(B(m_s[\eta_0],\frac{\delta}{2}))\cap \pi_s^{\perp})^{-1}\left(B(\zeta, \delta)\right)$ is of $0$ $\mathbbm{P}$-measure, which is in contradiction with the fact that $\mathbbm{P}$ is of full support as assumed in Hypothesis \ref{HypGaussProba}.
\end{prooff}
\begin{prooff}. of Proposition \ref{Prog}.
	
	We will proceed showing that for every bounded r.v. $Z$, 
	\begin{equation}\label{Eq8}
	(s,\eta)\longmapsto \mathbbm{E}^{s,\eta}[Z]\quad \text{ is }\mathbbm{F}^o\text{-progressively measurable.}
	\end{equation}
	
	We will make use of the functional monotone class Theorem.
	Let $\mathcal{C}$ be the set of r.v. of the type $e^{i\underset{j}{\sum}\lambda_jX^{i_j}_{t_j}},$ where
	$\lambda_j,j\leq n$
	are real numbers and 
$i_j,j\leq n$ belong to $\{1,\cdots,d\}$, 
	then  $\mathcal{C}$ is stable by product and
	generates the $\sigma$-algebra $\mathcal{F}$. Since the set of bounded r.v.
	$Z$ 
	verifying \eqref{Eq8} contains all constants and
	is closed by uniform convergence and by monotone pointwise convergence, then
	by the functional monotone class Theorem (see Theorem 21 in \cite{dellmeyer75} which easily extends to complex valued r.v.) it is enough to show that \eqref{Eq8} holds for all $Z\in\mathcal{C}$.
	
	We fix $n\in\mathbbm{N}$, $i_1,\cdots,i_n\in\{1,\cdots,d\}$ and  $t_1,\cdots,t_n\in[0,T]$.
	%
	For any $(s,\eta)$, $(X^{i_1}_{t_1},\cdots,X^{i_n}_{t_n})$ is under $\mathbbm{P}^{s,\eta}$ a Gaussian vector whose  mean is
	$$ \mu_{s,\eta}:=(m_s[\eta]^{i_1}(t_1),\cdots,m_s[\eta]^{i_n}(t_n))$$ and
its covariance matrix $\Sigma_s$, where its coefficient 
$\Sigma_s(k,j)$ is  equal to the $(i_k,i_j)$-th coefficient of $c^s(t_i,t_j)$. Therefore for all $\lambda\in\mathbbm{R}^n$ we have 
	\begin{equation}\label{EqFourier}
	\mathbbm{E}^{s,\eta}\left[e^{i\underset{j}{\sum}\lambda_jX^{i_j}_{t_j}}\right]=e^{i(\lambda,\mu_{s,\eta})-\frac{1}{2}(\lambda,\Sigma_s\lambda)}.
	\end{equation}
	By Corollary \ref{Coro_m}, $m_{\cdot}[\cdot](t)$ is $\mathbbm{F}^o$-progressively measurable for all $t$,
	and by Lemma \ref{Cs}, $s\mapsto \Sigma_s$ is a (deterministic) 
	continuous function. Therefore taking \eqref{EqFourier} into account, for all $\lambda\in\mathbbm{R}^n$, 
	$(s,\eta)\longmapsto \mathbbm{E}^{s,\eta}\left[e^{i\underset{j}{\sum}\lambda_jX^{i_j}_{t_j}}\right]$ is $\mathbbm{F}^o$-progressively measurable, meaning that \eqref{Eq8} holds with $Z\in\mathcal{C}$
	%
	%
	and the proof is complete.

\end{prooff}

Before the proof of  Proposition \ref{FiltRichtCont}, we need a few technical lemmas.
\begin{lemma}\label{Isometry}

Let $s \in [0,T]$.
	For any $n\in\mathbbm{N}^*$ and $t_1,\cdots,t_n\in [s,T]$,
 the joint law of $(\int_s^{t_j}k_i(t_j,r)dB_r)_{j\leq n;i\leq d}$ under $\mathbbm{P}$ is equal to the joint law of $(X^i_{t_j})_{j\leq n;i\leq d}$ under $\mathbbm{P}^s$.
\end{lemma}
\begin{proof}
	Since both laws relate to mean-zero Gaussian vectors, it is enough
 to check that the covariance matrices are the same. We pick some $i_1,i_2\leq d$ and $j_1,j_2\leq n$ and through Definition \ref{DefBasic}, Lemma \ref{Cs}, the following 
calculations hold: 
	\begin{equation}
	\begin{array}{rcl}
	\mathbbm{E}^s[X^{i_1}_{t_{j_1}}X^{i_2}_{t_{j_2}}]&=& c^s_{i_1,i_2}(t_{j_1},t_{j_2})\\
	&=&\int_s^{t_{j_1}\wedge t_{j_2}}k(t_{j_1},r)k(t_{j_2},r)dr\\
	&=&\mathbbm{E}\left[\int_s^{t_{j_1}}k(t_{j_1},r)dB_r \int_s^{ t_{j_2}}k(t_{j_2},r)dB_r\right],
	\end{array}
	\end{equation}
	which concludes the proof.
\end{proof}

\begin{corollary}\label{LinearIndep}
	
	For all $s\in[0,T]$, $1 \le i \le d$, every finite subfamily of\\
	$\left\{\int_s^tk_i(t,r)dB_r|i\leq d; t \in[0,T]\right\}$
	is  linearly independent in $L^2(\mathbbm{P})$, where for
	all $i\leq d$,
	$k_i$ denotes the $i$-th raw of $k$.
\end{corollary}
\begin{proof}
	
	We assume that $\underset{j\leq n,i\leq d}{\sum}\lambda_{j,i} \int_s^{t_j}k_{i}(t_j,r)
	dB_r = 0$ $\mathbbm{P}$ a.s. where $n\in\mathbbm{N}^*; \lambda_{i,j}\in\mathbbm{R}\,,j\leq n,i\leq d$; $t_1<\cdots< t_n\in[s,T]$. 
	
	By previous Lemma \ref{Isometry}, this is equivalent to having 
	\begin{equation}\label{EqLinInd}
 l:=	\underset{j\leq n,i\leq d}{\sum}\lambda_{j,i} X^i_{t_j} = 0\quad \mathbbm{P}^s\text{ a.s. }
	\end{equation}

	 We denote by  $l_s$ 
the restriction of $l$ in \eqref{EqLinInd} to  $\Omega_s^{\perp}$. 
	By \eqref{EqLinInd},  the non-empty open set ${l_s}^{-1}(\mathbbm{R}^*)$ is of zero $\mathbbm{P}^s$-measure, which is in contradiction with the fact that $supp(\mathbbm{P}^s)=\Omega_s^{\perp}$ thanks to Proposition \ref{FullSupp}.


	%
\end{proof}

\begin{lemma}\label{LemFilt}
	For $s\in[0,T]$ we recall that the spaces  $H^s(\mathbbm{P})$ and $H(\mathbbm{P})$ were introduced in Definition \ref{DefBasic}.
	Let $H^s(\mathbbm{P})^{\perp}$ us denote the orthogonal of $H^s(\mathbbm{P})$ in $H(\mathbbm{P})$. 
	There is an isometry  $\Phi^s$ mapping 
	$H^s(\mathbbm{P})^{\perp}$ onto $ H(\mathbbm{P}^s)$
	such that, for any $ s$, we have the following.
	\begin{enumerate}
		\item For $t \ge s $,
		$\Phi^s(H^s(\mathbbm{P})^{\perp}\cap H^t(\mathbbm{P}))=H^t(\mathbbm{P}^s)$.
		\item
		Let $B$ be the Brownian motion whose existence is assumed in item 3. of Hypothesis \ref{HypGaussProba}.
		We have 
		\begin{equation} \label{ESurj}
		\Phi^s\left(\int_s^tk_i(t,r)dB_r\right)=X^i_t, t\geq s, i\leq d.
		\end{equation}
\item		
$\Phi^s$ is the unique isometry fulfilling \eqref{ESurj}.
	\end{enumerate}
\end{lemma}

\begin{proof}
	We fix $s$ and denote by
	$p_s^{\perp}$ the orthogonal projection on $H^s(\mathbbm{P})^{\perp}$ from the 
	space $H(\mathbbm{P})$.
	
	For all $t\geq s$, $1 \le i\leq d$, under $\mathbbm{P}$ we have $X^i_t=\int_0^sk_i(t,r)dB_r+\int_s^tk_i(t,r)dB_r$,
	where $\int_0^sk_i(t,r)dB_r\in H^s(\mathbbm{P})$ and $\int_s^tk_i(t,r)dB_r$ is orthogonal to\\
	$\overline{Span(\{B^i_r|1 \le i\leq d;r\in[0,s]\})}^{L^2(\mathbbm{P})}$ (which by Remark \ref{RemHypGaussProba} is equal to $H^s(\mathbbm{P})$),
 hence belongs to $H^s(\mathbbm{P})^{\perp}$. So $p_s^{\perp}(X^i_t)=\int_s^tk_i(t,r)dB_r$. 
	
	Since $H(\mathbbm{P})=\overline{Span(\{X^i_r|i\leq d;r\in[0,T]\})}^{L^2(\mathbbm{P})}$ and
	$H^s(\mathbbm{P})^{\perp}=p_s^{\perp}(H(\mathbbm{P}))$ 
	then by continuity of $p_s^{\perp}$, 
	$$ 
	H^s(\mathbbm{P})^{\perp}=\overline{Span(\{p_s^{\perp}(X^i_r)|i\leq d;r\in[0,T]\})}^{L^2(\mathbbm{P})} = \overline V^{L^2(\mathbbm{P})},
	$$
	where 
	$$ V = Span\left\{\int_s^tk_i(t,r)dB_r|i\leq d;
	t \in[0,T]\right\}.$$
	We start by defining $\Phi^s$ on $V$.
	First we fix  $\Phi^s(\int_s^tk_i(t,r)dB_r)$
	conformally to \eqref{ESurj}.
	Since by Corollary \ref{LinearIndep}, $\left\{\int_s^tk_i(t,r)dB_r|i\leq d; t \in[0,T]\right\}$ is linearly independent, then we can extend $\Phi^s$  by linearity to the rest of $V$ without ambiguity.
	For all $t,u\geq s$, $i,j\leq d$, a consequence of Lemma \ref{Isometry} is that
	
	$$\mathbbm{E} \left[\int_s^tk_i(t,r)dB_r\int_s^uk_j(u,r)dB_r\right]=\mathbbm{E}^{s}[X^i_tX^j_u].$$
	This implies that $\Phi^s$ preserves the scalar product, and therefore that it is an isometry from 
	$V$ onto $Span(\{X^i_r|i\leq d;t\in[s,T]\})$, as a subset of $H(\mathbbm{P}^s)$. 
	
	By the extension theorem we can (uniquely) extend $\Phi^s$ to a continuous mapping from $H^s(\mathbbm{P})^{\perp}=\overline{Span(\{\int_s^tk_i(t,r)dB_r|i\leq d;t\in[s,T]\})}^{L^2(\mathbbm{P})}$ to 
	$$ H(\mathbbm{P}^s)=\overline{Span(\{X^i_r|i\leq d;t\in[s,T]\})}^{L^2(\mathbbm{P})}.$$ By continuity of the scalar product, $\Phi^s$ still preserves the scalar product, therefore the norm, and therefore is still injective. 
	The surjectivity follows by density of $Span(\{X^i_t|i\leq d;t\in[s,T]\})$ in $H(\mathbbm{P}^s)$ and \eqref{ESurj}.
	
The proof of items 2. are 3. is contained in previous constructing 
considerations.  
	It remains to show item 1. 
	We fix $t\in[s,T]$. We can argue as above
	and say that $H^t(\mathbbm{P})=\overline{Span(\{X^i_u|i\leq d;u\in[0,t]\})}^{L^2(\mathbbm{P})}$ and therefore that 
	$H^s(\mathbbm{P})^{\perp}\cap H^t(\mathbbm{P})$ is the closure of $Span(\{p_s^{\perp}(X^i_u)|i\leq d;u\in[0,t]\})=Span(\{\int_s^tk_i(u,r)dB_r|i\leq d;u\in[s,t]\})$ whose elements are mapped in $H^t(\mathbbm{P}^s)$ by $\Phi^s$  since $\Phi^s(\int_s^tk_i(u,r)dB_r)=X^i_u\in H^t(\mathbbm{P}^s)$ for all $u\in[s,t]$, $i\leq d$. $H^t(\mathbbm{P}^s)$ is closed and $\Phi^s$ is continuous, so $\Phi^s$ also maps 
	$$\overline{Span(\{\int_s^tk_i(u,r)dB_r|i\leq d;u\in[s,t]\})}^{L^2(\mathbbm{P})}=H^s(\mathbbm{P})^{\perp}\cap H^t(\mathbbm{P}),$$ into $H^t(\mathbbm{P}^s)$. Conversely, $(\Phi^s)^{-1}$ maps $X^i_u$ to $\int_s^tk_i(u,r)dB_r\in H^s(\mathbbm{P})^{\perp}\cap H^t(\mathbbm{P})$ for all $u\in[s,t]$, $i\leq d$. By continuity of $(\Phi^s)^{-1}$ (the inverse of an isometry being an isometry) and since $H^s(\mathbbm{P})^{\perp}\cap H^t(\mathbbm{P})$ is closed, $(\Phi^s)^{-1}$ also maps $\overline{Span(\{X^i_u|i\leq d;u\in[s,t]\})}^{L^2(\mathbbm{P})}=H^t(\mathbbm{P}^s)$ into $H^s(\mathbbm{P})^{\perp}\cap H^t(\mathbbm{P})$. This proves item 1.

\end{proof}

 In the proof of Proposition \ref{FiltRichtCont}, we will denote by $\mathcal{F}^{o,s,\eta}_t$, the $\sigma$-field 
$\mathcal{F}^o_t$ augmented with $\mathbbm{P}^{s,\eta}$-null sets.
\begin{prooff}. of Proposition \ref{FiltRichtCont}.
	
	By Notation \ref{Ps} we recall that under $\mathbbm{P}^{s,\eta}$, the law of
	$X^{s,\eta}=X-m_s[\eta]$ is $\mathbbm{P}^s$,
	where $m_s[\eta]$ is the deterministic function,
	introduced in Proposition \ref{CondExp} and Notation \ref{extension}.
	Consequently 
	$H(\mathbbm{P}^{s,\eta})=H(\mathbbm{P}^s)$.
	By Lemma \ref{LemFilt} that space is isometric to $H^s(\mathbbm{P})^{\perp}$. 
	We define 
	\begin{equation} \label{ETseta}
	T^{s,\eta}:\begin{array}{ccl}
	H(\mathbbm{P}^{s})&\longrightarrow& H(\mathbbm{P}^{s,\eta})\\
	Y&\longmapsto& Y-\mathbbm{E}^{s,\eta}[Y],
	\end{array} 
	\end{equation}
	which is clearly an isometry and maps $X^i_t$ to $X^{i,s,\eta}_t$ for all $t\geq s$, $i\leq d$, hence it is easy to show
	\begin{equation}\label{Tseta}
	T^{s,\eta}(H^t(\mathbbm{P}^{s}))=H^t(\mathbbm{P}^{s,\eta})\quad\text{ for all }t.
	\end{equation}
	We also define
	\begin{equation}\label{EPhiseta}
	\Phi^{s,\eta}:=T^{s,\eta}\circ\Phi_s:H^s(\mathbbm{P})^{\perp}\longrightarrow H(\mathbbm{P}^{s,\eta}),
	\end{equation}
	which is an isometry mapping $\int_s^tk_i(u,r)dB_r$ to $X^{i,s,\eta}_t$ for all $t\geq s$, $i\leq d$. Combining item 1. of Lemma \ref{LemFilt} and \eqref{Tseta}, it is clear that
	\begin{equation}\label{EPhiseta2}
	\Phi^{s,\eta}(H^s(\mathbbm{P})^{\perp}\cap H^t(\mathbbm{P}))=H^t(\mathbbm{P}^{s,\eta})\quad\text{ for all }t\geq s.
	\end{equation}
	
	By Remark \ref{RemHypGaussProba} item 2., $H(\mathbbm{P})$ contains all the
	r.v. related to $B$. Moreover  $B^j_t-B^j_s$ is for all $j\leq d$ and $t\geq s$ orthogonal to $\overline{Span(\{B^i_r|i\leq d;r\in[0,s]\})}^{L^2(\mathbbm{P})}$
	which is equal to $H^s(\mathbbm{P})$ again by Remark \ref{RemHypGaussProba} 2.
	
	So for all $t\geq s$, $i\leq d$, 
	\begin{equation} \label{EOrth}
	B^i_t-B^i_s\in H^s(\mathbbm{P})^{\perp}\cap H^t(\mathbbm{P})
	\end{equation}
	and 
	we denote $B^{i,s,\eta}_t:= \Phi^{s,\eta}(B^i_t-B^i_s)$ which by \eqref{EPhiseta2} verifies
	\begin{equation}\label{EOrth1}
	B^{i,s,\eta}_t\in 
	H^t(\mathbbm{P}^{s,\eta}).
	\end{equation}
	We also denote $B^{s,\eta}:=(B^{1,s,\eta},\cdots,B^{d,s,\eta})$. 
	$H(\mathbbm{P}^{s,\eta})$ is a Gaussian space 
	of mean-zero r.v. so clearly, $(B^{s,\eta}_t)_{t\in[s,T]}$ is a mean-zero Gaussian process. Taking into account the isometry $(\Phi^{s,\eta})^{-1}$, see
considerations after \eqref{EOrth},
 for all $s\leq r,t,u,v\leq T$, $i,j\leq d$, we have
	\begin{equation}
	\begin{array}{rcl}
	\mathbbm{E}^{s,\eta}[(B^{i,s,\eta}_v-B^{i,s,\eta}_u)(B^{j,s,\eta}_r-B^{j,s,\eta}_t)]&=&\langle B^{i,s,\eta}_v-B^{i,s,\eta}_u,B^{j,s,\eta}_r-B^{j,s,\eta}_t\rangle_{H(\mathbbm{P}^{s,\eta})}\\
	&=&\langle B^i_v-B^i_u,B^j_r-B^j_t\rangle_{H(\mathbbm{P})}\\
	&=&\mathbbm{E}[(B^i_v-B^i_u)(B^j_r-B^j_t)].
	\end{array}
	\end{equation}
	In particular, under $ \mathbbm{P}^{s,\eta}$, $(B^{s,\eta}_t)_{t\in[s,T]}$ has independent increments, and for all $s\leq t\leq u$, $B^{s,\eta}_u-B^{s,\eta}_t$ has variance $u-t$.
	
	By Kolmogorov's Theorem (see Theorem 2.8 and Problem 2.10 in \cite{ks}) $(B^{s,\eta}_t)_{t\in[s,T]}$ admits a  $\mathbbm{P}^{s,\eta}$ continuous version which is a  Brownian motion starting in $s$ and which we still denote $(B^{s,\eta}_t)_{t\in[s,T]}$. 
We prove below that
 this Brownian motion is adapted to $(\mathcal{F}^{o,s,\eta}_t)_{t\in[s,T]}$,
	that filtration being introduced just 
after the statement of the
	present Proposition \ref{FiltRichtCont}. 
Indeed, for all $t$, $i\leq d$, we have
	by \eqref{EOrth1} that
	$B^{i,s,\eta}_t\in H^t(\mathbbm{P}^{s,\eta})$. So $B^{i,s,\eta}_t$ is the $L^2(\mathbbm{P}^{s,\eta})$ limit of linear combinations of values of $X$ at times prior to $t$, hence it is $\mathcal{F}^{o,s,\eta}_t$-measurable as the $L^2(\mathbbm{P}^{s,\eta})$ limit of $\mathcal{F}^{o,s,\eta}_t$-measurable r.v.
	
	For all $i\leq d$, $t\geq s$,  $\int_s^t k_i(t,r)dB^{s,\eta}_r$ is, by linearity and continuity of $\Phi^{s,\eta}$ and by construction of the
	Wiener integral, equal to
	$\Phi^{s,\eta}(\int_s^tk_i(t,r)dB_r)$. Moreover we also have
 $\Phi^{s,\eta}(\int_s^tk_i(t,r)dB_r)= X^{i,s,\eta}_t$, see the lines after \eqref{EPhiseta}.  
	In particular $\int_s^t k_i(t,r)dB^{s,\eta}_r=X^{i,s,\eta}_t$ holds  in $L^2(\mathbbm{P}^{s,\eta})$ and the first statement of Proposition \ref{FiltRichtCont} is proved.
	
	In particular, $(X_t)_{t\geq s}$ is adapted to the filtration generated  
	by $(B^{s,\eta}_t)_{t\in[s,T]}$
	augmented with $\mathbbm{P}^{s,\eta}$-null sets.  Since $(B^{s,\eta}_t)_{t\in[s,T]}$ has been shown to be \\ $(\mathcal{F}^{o,s,\eta}_t)_{t\in[s,T]}$-adapted, then $(\mathcal{F}^{o,s,\eta}_t)_{t\in[s,T]}$ is equal to the filtration generated by $(B^{s,\eta}_t)_{t\in[s,T]}$ augmented with $\mathbbm{P}^{s,\eta}$-null sets, which is right-continuous since $(B^{s,\eta}_t)_{t\in[s,T]}$ is a Brownian Motion, see Proposition 7.7 in \cite{ks}.
	
	$(\mathcal{F}^{o,s,\eta}_t)_{t\in[s,T]}$ is therefore right-continuous. 
	In particular, for all $t\geq s$,
	$$ \mathcal{F}^{s,\eta}_t=\underset{\epsilon>0}{\bigcap}\mathcal{F}^{o,s,\eta}_{t+\epsilon}=\mathcal{F}^{o,s,\eta}_t.$$ This concludes the ''moreover'' statement
of Proposition  \ref{FiltRichtCont}.
	
\end{prooff}

\end{appendix}
{\bf ACKNOWLEDGEMENTS.} The work of the second named author
was partially supported by a public grant as part of the
{\it Investissement d'avenir project, reference ANR-11-LABX-0056-LMH,
  LabEx LMH,}
in a joint call with Gaspard Monge Program for optimization,
 operations research and their interactions with data sciences.

\bibliographystyle{plain}
\bibliography{biblioPhDBarrasso}

\end{document}